\documentclass[11pt,a4]{article}
\usepackage{latexsym}
\usepackage{amsmath}
\usepackage{amssymb}
\usepackage{amsthm}
\usepackage{amscd}
\usepackage{color}
\definecolor{dblue}{rgb}{0,0,0.45}
\definecolor{red}{rgb}{0.7,0,0} 

\newtheorem{theorem}{Theorem}[section]
\newtheorem{lemma}[theorem]{Lemma}
\newtheorem{corollary}[theorem]{Corollary}
\newtheorem{proposition}[theorem]{Proposition}

\theoremstyle{definition}

\newtheorem{remark}[theorem]{Remark}
\newtheorem{definition}[theorem]{Definition}
\newtheorem{example}[theorem]{Example}

\theoremstyle{remark}

\newcommand{\N}{{\mathbb N}}
\newcommand{\R}{{\mathbb R}}

\newcommand{\cG}{{\mathcal G}}

\def\Rn{{\mathbb{R}^n}}
\def\A{\mathbb A}
\def\i{\infty}
\def\Nn{\mathbb{N}_0}

\def\Lploc{L^p_{\rm loc}(\mathbb{R}^n)}
\def\Lloc{L^1_{\rm loc}(\mathbb{R}^n)}

\DeclareMathOperator*{\supp}{supp}

\begin{document}

\begin{center}
\LARGE Generalized Hardy-Morrey spaces
\end{center}

\centerline{\large Ali Akbulut}
\centerline{\it Department of Mathematics, Ahi Evran University, Kirsehir, Turkey}
\centerline{aakbulut@ahievran.edu.tr}

\vspace{3mm}
\centerline{\large Vagif S. Guliyev}
\centerline{\it Department of Mathematics, Ahi Evran University, Kirsehir, Turkey}
\centerline{\it Institute of Mathematics and Mechanics, Baku, Azerbaijan}
\centerline{vagif@guliyev.com}

%

\vspace{3mm}

\centerline{\large Takahiro Noi}
\centerline{\it Department of Mathematics and Information Science,
Tokyo Metropolitan University}
\centerline{taka.noi.hiro@gmail.com}

\vspace{3mm}

\centerline{\large Yoshihiro Sawano}
\centerline{\it Department of Mathematics and Information Science, Tokyo Metropolitan University}
\centerline{yoshihiro-sawano@celery.ocn.ne.jp}

\date{}

\begin{abstract}
The generalized Morrey space was defined  independetly 
by T. Mizuhara 1991 and E. Nakai  in 1994.
Generalized Morrey space ${\mathcal M}_{p,\phi}({\mathbb R}^n)$
is equipped with a parameter $0<p<\infty$ and a function
$\phi:{\mathbb R}^n \times (0,\infty) \to (0,\infty)$.
Our experience shows that ${\mathcal M}_{p,\phi}({\mathbb R}^n)$
is easy to handle when $1<p<\infty$.
However, when $0<p \le 1$,
the function space ${\mathcal M}_{p,\phi}({\mathbb R}^n)$
is difficult to handle as many examples show.

The aim of this paper is twofold.
One of them is to propose a way to deal with
${\mathcal M}_{p,\phi}({\mathbb R}^n)$
for $0<p \le 1$.
One of them is to propose here a way
to consider the decomposition method of generalized Hardy-Morrey spaces.
We shall obtain some estimates for these spaces
about the Hardy-Littlewood maximal operator.
Especially, the vector-valued estimates obtained in the earlier papers
are refined. The key tool is the weighted Hardy operator.
Much is known about the weighted Hardy operator.
Another aim
is to propose here a way
 to consider the decomposition method of generalized Hardy-Morrey spaces.
Generalized Hardy-Morrey spaces emerged
from generalized Morrey spaces.
By means of the grand maximal operator
and the norm of generalized Morrey spaces,
we can define generalized Hardy-Morrey spaces.
With this culmination, we can easily refine the existing results.
In particular, our results complement
the one the 2014 paper by Iida, the third author and Tanaka;
there was a mistake there.
As an application, we consider bilinear estimates,
which is the \lq \lq so-called" Olsen inequality.
\end{abstract}




\section{Introduction}
\noindent

In this paper,
we are concerned with generalized Hardy-Morrey spaces.
The generalized Morrey space ${\mathcal M}_{p,\phi}({\mathbb R}^n)$
is equipped with a function $\phi$ and a positive parameter $0<p<\infty$.
The generalized Morrey space ${\mathcal M}_{p,\phi}({\mathbb R}^n)$
was defined independetly 
by T. Mizuhara in 1991 \cite{Mi} and E. Nakai  in 1994 \cite{Nakai94}.
Although we can disprove $C^\infty_{\rm c}({\mathbb R}^n)$
is dense in ${\mathcal M}_{p,\phi}({\mathbb R}^n)$,
we are still able to develop
a theory of the function space
$H{\mathcal M}_{p,\phi}({\mathbb R}^n)$
called generalized Hardy-Morrey spaces.

Denote by ${\cG}_p$
the set of all the functions 
$\phi:{\mathbb R}^n \times (0,\infty) \to (0,\infty)$
decreasing in the second variable
such that $t\in(0,\infty) \mapsto t^{\frac{n}{p}}\phi(x,t)\in(0,\infty)$ is almost increasing
uniformly over the first variable $x$, so that there exists a constant $C>0$ such that
\[
\phi(x,r) \le \phi(x,s), \quad
C\phi(x,r)r^{n/p} \ge \phi(x,s)s^{n/p}
\]
for all $x \in {\mathbb R}^n$ and $0<s \le r<\infty$.
All \lq\lq cubes" in $\R^n$ are assumed
to have their sides parallel to the coordinate axes.
Denote by ${\mathcal Q}$ the set of all cubes.
For a cube $Q \in {\mathcal Q}$,
the symbol
$\ell(Q)$ stands for the side-length
of the cube $Q$;
$\ell(Q)\equiv|Q|^{\frac1n}$.
We denote by ${\mathcal Q}({\mathbb R}^n)$
the set of all cubes.
When we are given a cube $Q$,
we use the following abuse of notations
$\phi(Q) \equiv \phi(c(Q),\ell(Q))$,
where $c(Q)$ denotes the center of $Q$.

The generalized Morrey space
${\mathcal M}_{p,\phi}({\mathbb R}^n)$
is defined as the set of all measurable functions
$f$ for which the norm
\[
\|f\|_{{\mathcal M}_{p,\phi}}
\equiv
\sup_{Q \in {\mathcal Q}}
\frac{1}{\phi(Q)}
\left(\frac{1}{|Q|}\int_Q|f(y)|^p\,dy\right)^{\frac1p}
\]
is finite.

Observe that, if
$\phi(x,r)=r^{\frac{n}{p}}$,
then $\mathcal{M}_{p,\phi}(\mathbb{R}^n)=L^p(\mathbb{R}^n)$.
In the special case when $\phi(x,r) \equiv r^{\lambda/p-n/p}$,
we write
${\mathcal M}_{p,\lambda}({\mathbb R}^n)$
instead of
${\mathcal M}_{p,\phi}({\mathbb R}^n)$.
An important observation made by Nakai
is that we can assume that
$\phi$ itself is decreasing
and that
$\phi(t)t^{n/p} \le \phi(T)T^{n/p}$
for all $0<t \le T<\infty$
when $\phi$ is independent of $x$;
see \cite[p. 446]{Nakai00}.
Indeed, in the case when $1 \le p<\infty$,
Nakai established that there exists a function
$\rho$ such that
$\rho$ itself is decreasing,
that
$\rho(t)t^{n/p} \le \rho(T)T^{n/p}$
for all $0<t \le T<\infty$
and that
${\mathcal M}_{p,\phi}({\mathbb R}^n)
={\mathcal M}_{p,\rho}({\mathbb R}^n)$.
See \cite[(1.2)]{SaSuTa11-1} for the case when $0<p \le 1$.
The class ${\mathcal G}_p$
is defined to be the set of all $\phi$
such that
$\phi$ itself is decreasing
and that
$\phi(t)t^{n/p} \le \phi(T)T^{n/p}$
for all $0<t \le T<\infty$.
This assumption will turn out to be natural
even when $\phi$ depends on $x$;
see Section \ref{s2}.

One of the primary aims of this paper
is to prove the following decomposition
result about the functions in generalized Morrey spaces
${\mathcal M}_{1,\phi}({\mathbb R}^n)$:
\begin{theorem}\label{lem:131108-2}
Assume that
$\phi \in {\mathcal G}_1$ and $\eta \in {\mathcal G}_1$ satisfy
\begin{equation}\label{eq:131126-311}
\int_r^\infty \frac{\phi(x,s)}{\eta(x,s)s}\,ds
\le C
\frac{\phi(x,r)}{\eta(x,r)}.
\end{equation}
Assume that
$\{Q_j\}_{j=1}^\infty \subset {\mathcal Q}({\mathbb R}^n)$,
$\{a_j\}_{j=1}^\infty \subset {\mathcal M}_{1,\eta}({\mathbb R}^n)$
and
$\{\lambda_j\}_{j=1}^\infty \subset[0,\infty)$
fulfill
\begin{equation}\label{eq:thm1-1}
\|a_j\|_{{\mathcal M}_{1,\eta}} \le \frac{1}{\eta(\ell(Q_j))}, \quad
{\rm supp}(a_j) \subset Q_j, \quad
\left\|\sum_{j=1}^\infty \lambda_j \chi_{Q_j}
\right\|_{{\mathcal M}_{1,\phi}}
<\infty.
\end{equation}
Then
$\displaystyle f \equiv \sum_{j=1}^\infty \lambda_j a_j$
converges absolutely in $L^1_{\rm loc}({\mathbb R}^n)$ and satisfies
\begin{equation}\label{eq:thm1-2}
\|f\|_{{\mathcal M}_{1,\phi}}
\le C
\left\|\sum_{j=1}^\infty\lambda_j\chi_{Q_j}
\right\|_{{\mathcal M}_{1,\phi}}.
\end{equation}
\end{theorem}
The proof of this theorem is not so difficult
and it is given in an early stage of the present paper;
see Section \ref{s10},
where we do not use the Hardy-Littlewood maximal operator
as in \cite{IST14}.
Unlike the case when $p>1$,
when $0<p \le 1$,
${\mathcal M}_{p,\phi}({\mathbb R}^n)$ is a nasty space
as the following example shows.
\begin{example}
Denote by $M$ the Hardy-Littlewood maximal operator.
\begin{enumerate}
\item
Let $\eta:(0,\infty) \to (0,\infty)$ be a function
which is independent of the position $x$.
One defined ${\mathcal M}_{L\log L,\eta}$
by the following norm in \cite{SST12-2};
\[
\|f\|_{{\mathcal M}_{L\log L,\eta}}
\equiv
\sup_{Q \in {\mathcal Q}}
\frac{1}{\eta(\ell(Q))}
\inf\left\{
\lambda>0\,:\,
\int_Q
\frac{|f(x)|}{\lambda}
\log\left(e+\frac{|f(x)|}{\lambda}\right)\,dx \le |Q|
\right\}.
\]
In \cite[Lemma 3.5]{SST12-2},
we proved
$
C^{-1}
\|f\|_{{\mathcal M}_{L\log L,\eta}}
\le
\|Mf\|_{{\mathcal M}_{1,\eta}}
\le C
\|f\|_{{\mathcal M}_{L\log L,\eta}}
$
for all $f \in {\mathcal M}_{L\log L,\eta}({\mathbb R}^n)$.
\item
In \cite[Lemma 3.4]{SST12-2},
we proved
$
C^{-1}
\|f\|_{{\mathcal M}_{1,\eta}}
\le
\|Mf\|_{{\mathcal M}_{p,\eta}}
\le C
\|f\|_{{\mathcal M}_{1,\eta}}
$
for all $f \in {\mathcal M}_{1,\eta}({\mathbb R}^n)$.
\end{enumerate}
\end{example}
From these examples we see that
${\mathcal M}_{p,\phi}({\mathbb R}^n)$
with $p \in (0,1]$ is difficult to handle.
Probably, Theorem \ref{lem:131108-2} paves the way
to deal with such a nasty space.

Another method to handle these nasty spaces
to use the grand maximal operator
and
define generalized Hardy-Morrey spaces.
Let $t>0$ and $f \in L^1({\mathbb R}^n)$.
Then define the heat semigroup by:
\[
e^{t\Delta}f(x) \equiv
\int_{\mathbb{R}^n}
\frac{1}{\sqrt{(4\pi t)^n}}
\exp\left(-\frac{|x-y|^2}{4t}\right)f(y) dy
 \quad (x \in {\mathbb R}^n).
\]
In a well-known method using the duality,
we naturally extend $e^{t\Delta}f$
to the case when $f \in {\mathcal S}'({\mathbb R}^n)$.
Let $0<p \le 1$ and $\phi \in {\mathcal G}_p$.
The generalized Hardy-Morrey space
$H{\mathcal M}_{p,\phi}({\mathbb R}^n)$
is the set of all
$f \in {\mathcal S}'({\mathbb R}^n)$ 
satisfying
$\sup\limits_{t>0}|e^{t\Delta}f(\cdot)|
\in {\mathcal M}_{p,\phi}({\mathbb R}^n)$.
We equip
$H{\mathcal M}_{p,\phi}({\mathbb R}^n)$
with the following norm:
\begin{equation}\label{HM}
\|f\|_{H{\mathcal M}_{p,\phi}}
\equiv
\left\|\sup\limits_{t>0}|e^{t\Delta}f|\right\|_{{\mathcal M}_{p,\phi}}
\quad
(f \in H{\mathcal M}_{p,\phi}({\mathbb R}^n)).
\end{equation}

We define $d_p\equiv n/p-n$ for $0<p \le 1$.
In addition to Theorem \ref{lem:131108-2},
we shall prove the following two theorems in this paper.
\begin{theorem}\label{thm:131108-2}
Let $0<p \le 1$ and $d \ge d_p$.
Let $q$ satisfy
\begin{equation}\label{eq:140327-1111}
q \in [1,\infty] \cap (p,\infty].
\end{equation}
Assume that
$\phi,\eta \in {\mathcal G}_1$ satisfy
\begin{equation}\label{eq:140327-11156}
\int_r^\infty \frac{\phi(x,s)}{\eta(x,s)s}\,ds
\le C
\frac{\phi(x,r)}{\eta(x,r)}
\end{equation}
for $r>0$.
Assume in addition that
$\{Q_j\}_{j=1}^\infty \subset {\mathcal Q}({\mathbb R}^n)$,
$\{a_j\}_{j=1}^\infty \subset {\mathcal M}_{q,\eta}({\mathbb R}^n)$
and
$\{\lambda_j\}_{j=1}^\infty \subset[0,\infty)$
fulfill
\[
\left\|\left(
\sum_{j=1}^\infty (\lambda_j \chi_{Q_j})^p
\right)^{\frac1p}
\right\|_{{\mathcal M}_{p,\phi}}
<\infty
\]
and
\begin{equation}\label{eq:thm1-1-a}
\|a_j\|_{{\mathcal M}_{q,\eta}} \le
\frac{1}{\eta(Q_j)}, \quad
{\rm supp}(a_j) \subset Q_j, \quad
\int_{Q_j}a(x)x^\alpha\,dx=0
\end{equation}
for all $|\alpha| \le d$.
Then
$\displaystyle f\equiv \sum_{j=1}^\infty \lambda_j a_j$
converges in
${\mathcal S}'({\mathbb R}^n)$,
belongs to
$H{\mathcal M}_{p,\phi}({\mathbb R}^n)$
and satisfies
\begin{equation}\label{eq:thm1-2-b}
\|f\|_{H{\mathcal M}_{p,\phi}}
\le C
\left\|\left(
\sum_{j=1}^\infty (\lambda_j \chi_{Q_j})^p
\right)^{\frac1p}
\right\|_{{\mathcal M}_{p,\phi}}.
\end{equation}
\end{theorem}

\begin{example}
If there exist $u$ and $v$ with $v>u$ such that
$\eta(x,s)=s^{-n/v}$ and that
$\phi(x,s)s^{n/u} \le \phi(x,r)r^{n/u}$
for all $s$ and $r$ with $s \ge r$,
then (\ref{eq:140327-1115}) is satisfied.
\end{example}

Theorem \ref{thm:131108-2} will refine \cite[p. 100 Theorem]{JW}
in that we can postulate a weaker integrability condition on $a_j$
in Theorem \ref{thm:131108-2}.
We shall take its advantage in Section \ref{s5}.

\begin{theorem}\label{thm:2}
Let $L \in {\mathbb N} \cup \{0\}.$
Let $0<p\le 1$ and $f \in H{\mathcal M}_{p,\phi}({\mathbb R}^n)$.
Then under 
\[
\int_r^\infty \frac{\phi(x,s)^p}{\eta(x,s)^ps}\,ds
\le C
\frac{\phi(x,r)^p}{\eta(x,r)^p}
\]
and 
\[
\int_r^\infty \phi(x,s)\frac{ds}{s} \le C
\phi(x,r).
\]
and $\phi \in {\cG}_1$,
there exists a triplet
$\{\lambda_j\}_{j=1}^\infty \subset [0,\infty)$,
$\{Q_j\}_{j=1}^\infty \subset {\mathcal Q}({\mathbb R}^n)$
and
$\{a_j\}_{j=1}^\infty \subset L^\infty({\mathbb R}^n)$
such that
$f=\sum_{j=1}^\infty \lambda_j a_j$
in ${\mathcal S}'({\mathbb R}^n)$ and that,
for all $v>0$
\begin{equation}\label{eq:thm2-1}
|a_j| \le \chi_{Q_j}, \quad
\int_{{\mathbb R}^n}x^\alpha a_j(x)\,dx=0, \quad
\left\|\left(\sum_{j=1}^\infty
(\lambda_j\chi_{Q_j})^{v}
\right)^{1/v}\right\|_{{\mathcal M}_{p,\phi}}
\le C_v\|f\|_{H{\mathcal M}_{p,\phi}}
\end{equation}
for all multi-indices $\alpha$ with $|\alpha| \le L$.
Here the constant $C_v>0$ is independent of $f$.
\end{theorem}

It is not known that
$H{\mathcal M}_{p,\phi}({\mathbb R}^n)
\cap L^1_{\rm loc}({\mathbb R}^n)$
is dense
in $H{\mathcal M}_{p,\phi}({\mathbb R}^n)$.
It seems that 
$H{\mathcal M}_{p,\phi}({\mathbb R}^n)
\cap L^1_{\rm loc}({\mathbb R}^n)$
is not dense
in $H{\mathcal M}_{p,\phi}({\mathbb R}^n)$
as the fact that
${\mathcal M}_{p,\phi}({\mathbb R}^n)
\cap L^\infty_{\rm comp}({\mathbb R}^n)$
is not dense
in ${\mathcal M}_{p,2p}({\mathbb R}^n)$ implies.
Recall that in \cite{Stein1993},
we resorted to density
of $H^p({\mathbb R}^n) \cap L^1_{\rm loc}({\mathbb R}^n)$
to obtain the atomic decomposition
of $H^p({\mathbb R}^n)$.
The difficulty willl cause a disability;
we can prove Theorem \ref{thm:2} only when
$f \in H{\mathcal M}_{p,\phi}({\mathbb R}^n)$.
By using a diagonal argument,
we circumbent this problem;
see (\ref{eq:150302-141}) and (\ref{eq:150302-14}).

Before we go further, let us recall some special cases
related to generalized Morrey spaces.
\begin{example}\
\begin{enumerate}
\item
Generalized Morrey spaces can cover
$L^\infty({\mathbb R}^n)$ spaces
by letting $\phi \equiv 1$.
\item{\rm \cite[Theorem 5.1]{SW}}\label{p1.3}
Let $1<p<\infty$ and $0<\lambda<n$.
Then there exists a positive constant $C_{p,\lambda}$ such that
\begin{equation}\label{eq:150308-1}
\int_B|f(x)|dx\leq C_{p,\lambda}|B|(1+|B|)^{-\frac{1}{p}}
\log\left(e+\frac{1}{|B|}\right)
\|(1-\Delta)^{\lambda/2p}f\|_{L^{p,\lambda}}
\end{equation}
holds for all $f\in L^{p,\lambda}(\R^n)$
with $(1-\Delta)^{\lambda/2p}f \in L^{p,\lambda}(\R^n)$
and for all balls $B$.
See \cite[Section 5]{EGNS} for more details.
In view of the integral kernel of $(1-\Delta)^{-\alpha/2}$
(see \cite{Stein}) and the Adams theorem,
we have
\begin{equation}\label{1.7}
(1-\Delta)^{-\alpha/2}:
L^{p,\lambda}(\R^n) \to L^{q,\lambda}(\R^n)
\end{equation}
is bounded as long as the parameters $p,q,\lambda$ and $\alpha$ satisfy
\[
1 < p,q<\infty, \,
0<\lambda \le n, \,
\frac{1}{q}=\frac{1}{p}-\frac{\alpha}{\lambda}.
\]
However, if $\alpha=\frac{\lambda}{p}$,
the number $q$ not being finite,
the boundedness assertion $(\ref{1.7})$ is no longer true.
Hence $(\ref{eq:150308-1})$ can be considered as a substitute
of $(\ref{1.7})$.

\end{enumerate}
\end{example}

A passage to the Hardy type space from a given function space
is not a mere quest to generality.
Many people have shown that
Hardy spaces $H^p(\mathbb{R}^n)$ ($0<p\leq \infty$)
can be more informative than
Lebesgue spaces $L^p(\mathbb{R}^n)$
when we discuss the boundedness of some operators.
For example, the Riesz transform is bounded
from $H^1(\mathbb{R}^n)$ to $L^1({\mathbb R}^n)$,
although they are not bounded on $L^1(\mathbb{R}^n)$.
One of the earliest real variable definitions of Hardy spaces
were based on the grand maximal operator,
which is discussed in \cite{Stein1993} and references therein.
One can also give an equivalent definition
for Hardy spaces by using the atomic decomposition.
This definition states that
any elements of Hardy spaces can be represented as the series of atoms.
An atom is a compactly supported function
which enjoys the size condition and the cancellation moment condition.
One of the advantages of the atomic decompositions in Hardy spaces
is that we can prove the boundedness of some operators can be verified
only for the collection of atoms.
The concept of the atomic decomposition in Hardy spaces
can be developed to other function spaces.
Some of these works are
the decomposition of Hardy--Morrey spaces \cite{JW},
the decomposition of Hardy spaces with variable exponent \cite{Sa13},
and
the atomic decomposition of Morrey spaces \cite{IST14}.
Motivated by these advantages that Hardy spaces enjoy,
in our current research,
we investigate the atomic decomposition
for generalized Hardy-Morrey spaces,
where we are based on the definition by means of the grand maximal operator.

There are many attempts of obtaining non-smooth atomic decompositions
by using the grand maximal operator
\cite{GHS-pre,IST14,NaSa2012,NaSa-pre2013},
where the authors handled Morrey spaces,
Orlicz spaces and variable exponent Lebesgue spaces.
Unlike Orlicz spaces, variable exponent Lebesgue spaces,
in general we can take a sequence $\{f_j\}_{j=1}^\infty$
of functions such that
\[
f_1 \ge f_2 \ge \cdots \ge f_j \ge f_{j+1} \ge \cdots \to 0, \quad
\inf_{j \in {\mathbb N}}\|f_j\|_{{\mathcal M}_{p,\phi}}>0.
\]
For example, when $0<p<a$,
the sequence
$f_j(x)\equiv\chi_{(j,\infty)}(|x|)|x|^{-n/a}$ does the job.
This makes it more difficult
to look for a good dense space
of ${\mathcal M}_{p,n-a}({\mathbb R}^n)$.
This difficulty prevents us from using (\ref{eq:150302-1})
directly.

We adopt the following notations:
\begin{enumerate}
\item
$\N_0\equiv \{0,1,\ldots\}$.
\item
Let $A,B \ge 0$.
Then $A \lesssim B$ means
that there exists a constant $C>0$
such that $A \le C B$
and
$A \thickapprox B$ stands for
$A \lesssim B \lesssim A$,
where $C$ depends only on the parameters
of importance.
\item
By \lq \lq cube" we mean a compact cube
whose edges are parallel to the coordinate {\it axes}.
The metric ball defined by $\ell^\infty$ is called a {\it cube}.
If a cube has center $x$ and radius $r$,
we denote it by $Q(x,r)$.
{}From the definition of $Q(x,r)$,
its volume is $(2r)^n$.
We write $Q(r)$ instead of $Q({\rm o},r)$,
where ${\rm o}$ denotes the origin.
Given a cube $Q$,
we denote by $c(Q)$ {\it the center of $Q$}
and by $\ell(Q)$ the {\it sidelength of $Q$}:
$\ell(Q)=|Q|^{1/n}$,
where $|Q|$ denotes the volume of the cube $Q$.
\item
Given a cube $Q$ and $k>0$,
$k\,Q$ means {\it the cube concentric to $Q$
with sidelength $k\,\ell(Q)$}.
\item
By a {\it dyadic cube},
we mean a set of the form $2^{-j}m+[0,2^{-j}]^n$
for some $m \in {\mathbb Z}^n$ and $j \in {\mathbb Z}$.
The set of all dyadic cubes will be denoted
by ${\mathcal D}$.
\item
Let $\mathcal{Q}_x(\mathbb{R}^n)$ be a collection
of all cubes that contain $x\in \mathbb{R}^n$.
\item
In the whole paper,
we adopt the following definition of the Hardy-Littlewood maximal operator
to estimate some integrals.
The Hardy-Littlewood maximal operator $M$ is defined by
\begin{equation}\label{eq:140205-101}
Mf(x)\equiv\sup\limits_{Q\in \mathcal{Q}_x (\mathbb{R}^n)}
\frac{1}{|Q|} \int_{Q} |f(y)| dy,
\end{equation}
for a locally integrable function $f$.
\item
Let $0<\alpha<n$.
We define the fractional integral operator
 $I_{\alpha}$
by
\begin{align*}
I_{\alpha}f(x)\equiv\int_{\mathbb{R}^n} \frac{f(y)}{|x-y|^{n-\alpha}} \ dy
\end{align*}
for all suitable functions $f$ on $\R^n$.
\item
Let $0<q\le \infty$.
If $F\equiv \{ f_j\}_{j=-\infty}^{\infty}$
is a sequence of complex-valued Lebesgue measurable functions on $\Rn$ such that
$$
\|F \|_{{\mathcal M}_{p,\phi}(l_{q})}= \| \|F\|_{l_{q}} \|_{{\mathcal M}_{p,\phi}}<\infty,
$$
write $F \in {\mathcal M}_{p,\phi}(l_{q},{\mathbb R}^n)$.
\item
Let $0 < p, \, q \le \infty$.
If $\{f_j\}_{j \in \Nn}$
is a sequence of complex-valued Lebesgue
measurable functions on $\Omega \subseteq \Rn$,
then define;
\begin{gather*}
\left\| \left\{f_j \right\}_{j \in \Nn}\right\|_{l_q(L^p(\Omega))}
\equiv
\left\|\left\{\left\| f_j \right\|_{L^p(\Omega)}\right\}_{j \in \Nn}
\right\|_{l_q}
\end{gather*}
and
\begin{gather*}
\left\| \left\{f_j \right\}_{j \in \Nn}
\right\|_{L^p(\Omega)(l_q)}
\equiv
\left\|\left\| \left\{f_j \right\}_{j \in \Nn} \right\|_{l_q} \right\|_{L^p(\Omega)}.
\end{gather*}
The space ${\mathcal M}_{p,\phi}(l_{q},{\mathbb R}^n)$ stands for the set of all sequences
$\{ f_j \}_{j \in \N_{0}}$ of complex-valued Lebesgue measurable functions
on $\Rn$ for which
$$
\left\| \{ f_j \}_{j \in \Nn} \right\|_{{\mathcal M}_{p,\phi}(l_{q})}
\equiv
\left\| \left\| \{ f_j\}_{j \in \Nn} \right\|_{l_{q}}\right\|_{{\mathcal M}_{p,\phi}}
<\infty.
$$
Similarly denote by $W{\mathcal M}_{p,\phi}(l_{q},{\mathbb R}^n)$
the set of all sequences
$\{ f_j \}_{j \in \N_{0}}$ for which
$$
\left\| \{ f_j \}_{j \in \Nn} \right\|_{W{\mathcal M}_{p,\phi}(l_{q})}
\equiv
\left\| \left\| \{ f_j \}_{j \in \Nn} \right\|_{l_{q}}
\right\|_{W{\mathcal M}_{p,\phi}}
<\infty.
$$
The spaces
$l_{q}\left({\mathcal M}_{p,\phi}({\mathbb R}^n) \right)$
and
$l_{q}\left(W{\mathcal M}_{p,\phi}({\mathbb R}^n) \right)$
can be also
defined similarly by the norms;
$$
\left\| \{ f_j \}_{j \in \Nn} \right\|_{l_{q}\left({\mathcal M}_{p,\phi} \right)}
\equiv
\left\| \left\{ \left\| f_j\right\|_{{\mathcal M}_{p,\phi}}
\right\}_{j \in \Nn} \right\|_{l_{q}}
<\infty
$$
and
$$
\left\| \{ f_j \}_{j \in \Nn}
\right\|_{l_{q}\left(W{\mathcal M}_{p,\phi} \right)}
\equiv
\left\| \left\{ \left\| f_j\right\|_{W{\mathcal M}_{p,\phi}}
\right\}_{j \in \Nn} \right\|_{l_{q}}
<\infty,
$$
respectively.
\end{enumerate}

Finally, to conclude this section,
we briefly describe how we organize the remaining part of this paper.
Sections \ref{s2} collects preliminary facts.
We collect some elementary facts on function spaces
and investigate the Hardy-Littlewood maximal operator
in Section \ref{s2}.
We prove the main theorems in Section \ref{s4}
and apply these main theorems in Section \ref{s5}.

\section{Fundamental structure of function spaces}
\label{s2}

\subsection{Structure of generalized Morrey spaces}

\begin{lemma}
Let 
$\phi:{\mathbb R}^n \times (0,\infty) \to (0,\infty)$
be a function and let $0<p<\infty$.
Then there exists a function 
$\psi:{\mathbb R}^n \times (0,\infty) \to (0,\infty)$
satisfying
\begin{equation}\label{eq:150311-1}
\psi(y,s)s^{n/p} \le \psi(x,r)r^{n/p}
\end{equation}
for all $x,y \in {\mathbb R}^n$ and $r,s>0$ 
with $\|x-y\|_\infty \le r-s$
such that
${\mathcal M}_{p,\phi}({\mathbb R}^n) =
{\mathcal M}_{p,\psi}({\mathbb R}^n)$
with norm coincidence.
\end{lemma}

\begin{proof}
Let us set
\begin{equation}\label{eq:150311-3}
\psi(x,r)\equiv
\inf_{y \in {\mathbb R}^n}
\left(
\inf_{v \ge r+\|x-y\|_\infty}\phi(y,v)
\left(\frac{v}{r}\right)^{n/p}
\right) \quad (x \in {\mathbb R}^n, r>0).
\end{equation}
Then we have
$\phi(x,r) \ge \psi(x,r)$ trivially and hence
$\|f\|_{{\mathcal M}_{p,\phi}}
\le
\|f\|_{{\mathcal M}_{p,\psi}}$.
Meanwhile,
\begin{align*}
\|f\|_{{\mathcal M}_{p,\psi}}
&=
\sup_{Q \in {\mathcal Q}}
\frac{1}{\psi(c(Q),\ell(Q))}
\left(\frac{1}{|Q|}\int_Q|f(y)|^p\,dy\right)^{\frac1p}\\
&=
\sup_{Q \in {\mathcal Q}}\sup_{y \in {\mathbb R}^n}
\left(
\sup_{v \ge \ell(Q)+\|x-y\|_\infty}
\frac{1}{\phi(y,v)}
\left(\frac{1}{v^n}\int_Q|f(y)|^p\,dy\right)^{\frac1p}
\right)\\
&\le
\sup_{y \in {\mathbb R}^n}
\left(
\sup_{v \ge \ell(Q)+\|x-y\|_\infty}
\frac{1}{\phi(y,v)}
\left(\frac{1}{v^n}\int_{Q(y,v)}|f(y)|^p\,dy\right)^{\frac1p}
\right)\\
&=\|f\|_{{\mathcal M}_{p,\phi}}.
\end{align*}
Thus, we have
${\mathcal M}_{p,\phi}({\mathbb R}^n)=
{\mathcal M}_{p,\psi}({\mathbb R}^n)$
with norm coincidence.

From the definition of $\psi$,
it is easy to check that we have (\ref{eq:150311-1}).
\end{proof}

\begin{lemma}
Let $0<p<\infty$ and let 
$\phi:{\mathbb R}^n \times (0,\infty) \to (0,\infty)$
be a function satisfying $(\ref{eq:150311-1})$.
Then there exists a function 
$\psi:{\mathbb R}^n \times (0,\infty) \to (0,\infty)$
satisfying 
\begin{equation}\label{eq:150311-2}
\psi(x,r) \le \psi(x,s)
\end{equation}
for all $x \in {\mathbb R}^n$ and $0<s \le r<\infty$
such that
${\mathcal M}_{p,\phi}({\mathbb R}^n) =
{\mathcal M}_{p,\psi}({\mathbb R}^n)$
with norm equivalence.
\end{lemma}

\begin{proof}
Let us set
\[
\psi(x,r) \equiv
\inf_{0<s \le r}\left(\sup_{y \in Q(x,r)}\phi(y,s)\right).
\]
It is easy to see that (\ref{eq:150311-2}) is satisfied.
Then 
\[
\psi(x,r)\le \sup_{y \in Q(x,r)}\phi(y,r)
\le 3^{n/p}\sup_{y \in Q(x,r)}\phi(y,3r)
\] 
from (\ref{eq:150311-1})
and hence
$\|f\|_{{\mathcal M}_{p,\phi}}\le 3^{n/p}\|f\|_{{\mathcal M}_{p,\psi}}$.
Meanwhile,
for all $Q \in {\mathcal Q}$ and $0<s \le \ell(Q)$,
we can find a cube $R=R_Q(s)$ contained in $Q$ 
such that $\ell(R)=s$ and that
\[
\left(\frac{1}{|R|}\int_{R}|f(y)|^p\,dy\right)^{1/p}
\ge 2^{-n/p}
\left(\frac{1}{|Q|}\int_{Q}|f(y)|^p\,dy\right)^{1/p}.
\]
Therefore, it follows that
\begin{align*}
\|f\|_{{\mathcal M}_{p,\psi}}
&=
\sup_{Q \in {\mathcal Q}}
\frac{1}{\psi(c(Q),\ell(Q))}
\left(\frac{1}{|Q|}\int_Q|f(y)|^p\,dy\right)^{\frac1p}\\
&=
\sup_{Q \in {\mathcal Q}}
\sup_{0<s<r}
\left(
\inf_{y \in Q}
\frac{1}{\phi(y,s)}
\left(\frac{1}{|Q|}\int_{Q}|f(y)|^p\,dy\right)^{\frac1p}
\right)\\
&\le 2^{n/p}
\sup_{Q \in {\mathcal Q}}
\sup_{0<s<r}
\left(
\inf_{y \in Q}\frac{1}{\phi(y,s)}
\left(\frac{1}{|R_Q(s)|}\int_{R_Q(s)}|f(y)|^p\,dy\right)^{\frac1p}
\right)\\
&\le 2^{n/p}
\sup_{Q \in {\mathcal Q}}
\sup_{0<s<r}
\left(
\frac{1}{\phi(R_Q(s))}
\left(\frac{1}{|R_Q(s)|}\int_{R_Q(s)}|f(y)|^p\,dy\right)^{\frac1p}
\right)\\
&\le
2^{n/p}\|f\|_{{\mathcal M}_{p,\phi}},
\end{align*}
as was to be shown.
\end{proof}

\begin{lemma}\label{lem:150311-1}
Let $0<p<\infty$ and let 
$\phi:{\mathbb R}^n \times (0,\infty) \to (0,\infty)$
be a function satisfying 
\begin{equation}\label{eq:150311-11}
\phi(x,r) \le \phi(x,s)
\end{equation}
for all $0<s \le r<\infty$ and $x \in {\mathbb R}^n$.
Then there exists a function 
$\psi:{\mathbb R}^n \times (0,\infty) \to (0,\infty)$
satisfying $(\ref{eq:150311-1})$ and $(\ref{eq:150311-2})$
such that
${\mathcal M}_{p,\phi}({\mathbb R}^n) =
{\mathcal M}_{p,\psi}({\mathbb R}^n)$
with norm coincidence.
\end{lemma}

\begin{proof}
Let us define $\psi$ by (\ref{eq:150311-3}).
Then as we have seen,
$\psi:{\mathbb R}^n \times (0,\infty) \to (0,\infty)$
satisfies $(\ref{eq:150311-1})$
and
${\mathcal M}_{p,\phi}({\mathbb R}^n) =
{\mathcal M}_{p,\psi}({\mathbb R}^n)$
with norm coincidence.
It remains to check (\ref{eq:150311-2}).
Let $R<R'$.
Then,
from (\ref{eq:150311-11}),
we obtain
\begin{align*}
\psi(x,R')
&=
\inf_{y \in {\mathbb R}^n}
\left(
\inf_{v \ge R'+\|x-y\|_\infty}\phi(y,v)
\left(\frac{v}{R'}\right)^{n/p}
\right)\\
&\le
\inf_{y \in {\mathbb R}^n}
\left(
\inf_{v \ge R'+R'\|x-y\|_\infty/R}\phi(y,v)
\left(\frac{v}{R'}\right)^{n/p}
\right)\\
&=
\inf_{y \in {\mathbb R}^n}
\left(
\inf_{v \ge R+\|x-y\|_\infty}\phi(y,R'v/R)
\left(\frac{v}{R}\right)^{n/p}
\right)\\
&\le
\inf_{y \in {\mathbb R}^n}
\left(
\inf_{v \ge R+\|x-y\|_\infty}\phi(y,v)
\left(\frac{v}{R}\right)^{n/p}
\right)=\psi(x,R).
\end{align*}
This proves (\ref{eq:150311-2}).
\end{proof}

The following compatibility condition:
\begin{equation}\label{eq:compatibility}
\phi(x,r) \sim \phi(y,r) \quad (|x-y| \le r)
\end{equation}
can be naturally postulated.
\begin{proposition}
Let $0<p<\infty$ and let 
$\phi:{\mathbb R}^n \times (0,\infty) \to (0,\infty)$
be a function
satisfying $(\ref{eq:150311-1})$ and $(\ref{eq:150311-2})$.
Then $\phi$ satisfies $(\ref{eq:compatibility})$.
\end{proposition}

\begin{proof}
By $(\ref{eq:150311-2})$,
we have
$\phi(x,r) \gtrsim \phi(x,3r)$
and by $(\ref{eq:150311-1})$
$\phi(x,3r) \gtrsim \phi(y,r)$.
\end{proof}

With Lemma \ref{lem:150311-1} in mind,
we always assume that $\phi \in \cG_p$ satisfies (\ref{eq:150311-1}).

The main structure of this generalized Morrey space
${\mathcal M}_{p,\phi}({\mathbb R}^n)$ is as follows:
\begin{proposition}
Let $0<p<\infty$ and $\phi \in {\mathcal G}_p$.
Assume $(\ref{eq:150311-1})$.
Then
\begin{equation}\label{eq:141006-1}
\frac{1}{\phi(Q)}
\le
\|\chi_{Q}\|_{{\mathcal M}_{p,\phi}}
\le C
\frac{1}{\phi(Q)}.
\end{equation}
\end{proposition}

\begin{proof}
By the definition,
\[
\|\chi_Q\|_{{\mathcal M}_{p,\phi}}
=
\sup_{R \in {\mathcal Q}}\frac{1}{\phi(R)}\left(\frac{|Q \cap R|}{|R|}\right)^{1/p}.
\]
Thus, the left inequality is clear.
From (\ref{eq:150311-1}),
we have
\[
\|\chi_Q\|_{{\mathcal M}_{p,\phi}}
=
\sup_{R \in {\mathcal Q}, Q \setminus 3R \ne \emptyset}\le
\frac{1}{\phi(R)}\left(\frac{|Q \cap R|}{|R|}\right)^{1/p}.
\]
Let $R$ be a cube such that $3R$ does not engulf $Q$
and that $R$ intersects $Q$. 
We let $S$ be a cube concentric to $R$
having sidelength $3Q$.
Then 
\[
\frac{1}{\phi(R)}\left(\frac{|Q \cap R|}{|R|}\right)^{1/p}
\le 
\frac{1}{\phi(S)}
\le 
\frac{C}{\phi(Q)}.
\]
Thus, we obtain the right inequality.
\end{proof}

\begin{remark}
See \cite[Proposition 2.1]{GHS-pre} for the case when $p \ge 1$.
The same proof works for this case
but for the sake of convenience for readers
we supply the whole proof.
\end{remark}

\begin{corollary}\label{cor:150308-1}
Let $0<p \le 1$ and $\phi \in {\mathcal G}_p$.
There exists $N \gg 1$ such that
$(1+|\cdot|)^{-N} \in {\mathcal M}_{p,\phi}({\mathbb R}^n)$.
In particular,
${\mathcal M}_{1,\phi}({\mathbb R}^n)$
is continously embedded into ${\mathcal S}'({\mathbb R}^n)$.
\end{corollary}

\begin{proof}
Just observe that each term in (\ref{eq:141006-1}) grows polynomially.
\end{proof}

Prior to the proof of Theorems \ref{lem:131108-2} and \ref{thm:2},
observe that we have the following equivalent expression:
\[
\|f\|_{{\mathcal M}_{p,\phi}}
\sim
\sup_{Q \in {\mathcal D}}
\frac{1}{\phi(Q)}
\left(\frac{1}{|Q|}\int_Q|f(y)|^p\,dy\right)^{\frac1p}.
\]

\subsection{Proof of Theorem \ref{lem:131108-2}}
\label{s10}

By replacing $a_j$ with $|a_j|$ if necessary,
we may assume that $a_j$ is non-negative.

Let $Q$ be a fixed dyadic cube.
We need to show
\begin{equation}\label{eq:140327-1000}
\frac{1}{\phi(Q)|Q|}\int_Q|f(y)|\,dy
\lesssim
\left\|\sum_{j=1}^\infty\lambda_j\chi_{Q_j}
\right\|_{{\mathcal M}_{1,\phi}}.
\end{equation}
By considering dyadic cubes of equivalent length,
we may assume that $Q_j$ is dyadic.
Let us set
\[
J_1\equiv\{j\,:\,Q_j \subset Q\}, \quad
J_2\equiv\{j\,:\,Q_j \supset Q\}.
\]
In terms of the sets $J_1$ and $J_2$,
we shall show
\begin{equation}\label{eq:140327-101}
\frac{1}{\phi(Q)|Q|}\int_Q
\sum_{j \in J_1} \lambda_j a_j(y)\,dy
\lesssim
\left\|\sum_{j=1}^\infty\lambda_j\chi_{Q_j}
\right\|_{{\mathcal M}_{1,\phi}}
\end{equation}
and
\begin{equation}\label{eq:140327-102}
\frac{1}{\phi(Q)|Q|}\int_Q
\sum_{j \in J_2} \lambda_j a_j(y)\,dy
\lesssim
\left\|\sum_{j=1}^\infty\lambda_j\chi_{Q_j}
\right\|_{{\mathcal M}_{1,\phi}}.
\end{equation}
Once we prove (\ref{eq:140327-101})
and (\ref{eq:140327-102}),
then we will have proved
(\ref{eq:140327-1000}).

To prove (\ref{eq:140327-101}),
we observe
\[
\frac{1}{|Q_j|}\int_{Q_j}a_j(x)\,dx
\le
\eta(Q_j)\|a_j\|_{{\mathcal M}_{1,\eta}}
\le
1
\]
and that
\begin{align*}
\frac{1}{\phi(Q)|Q|}\int_Q
\sum_{j \in J_1} \lambda_j a_j(y)\,dy
&=
\frac{1}{\phi(Q)|Q|}\int_Q
\sum_{j \in J_1} \lambda_j
\left(\frac{1}{|Q_j|}\int_{Q_j}a_j(y)\,dy\right)
\chi_{Q_j}(z)\,dz\\
&\le
\frac{1}{\phi(Q)|Q|}\int_Q
\sum_{j \in J_1} \lambda_j
\chi_{Q_j}(z)\,dz\\
&\le
\left\|\sum_{j=1}^\infty\lambda_j\chi_{Q_j}
\right\|_{{\mathcal M}_{1,\phi}}.
\end{align*}

To prove (\ref{eq:140327-102}),
we note that there exists an increasing (possibly finite) sequence
of dyadic cubes $R_1,R_2,\ldots$ such that
$\{Q_j \,:\,j \in J_2\}
=
\{R_1,R_2,\ldots\}.$
By using this sequence and (\ref{eq:131126-311}), we have
\[
\int_Q
\sum_{j \in J_2}\frac{\lambda_j a_j(y)}{\phi(Q)|Q|}\,dy
=
\sum_m
\lambda_m
\frac{\eta(Q)\phi(R_m)}{\phi(Q)\eta(R_m)}
\|\chi_{R_m}\|_{{\mathcal M}_{1,\phi}}
\lesssim
\left\|\sum_{j=1}^\infty\lambda_j\chi_{Q_j}
\right\|_{{\mathcal M}_{1,\phi}},
\]
as was to be shown.

\begin{remark}
It may be interesting to compare Theorem \ref{lem:131108-2}
with \cite[Theorem 1.1]{IST14}.
We state in words of ${\mathcal M}_{p,\lambda}({\mathbb R}^n)$:
Suppose that the parameters $q,t,\lambda,\rho$ satisfy
$$
1<q<\infty, \quad 1<t<\infty, \quad
q<t, \quad \frac{q}{n-\lambda}<\frac{t}{n-\rho}.
$$
Assume that
$\{Q_j\}_{j=1}^\infty \subset {\mathcal Q}({\mathbb R}^n)$,
$\{a_j\}_{j=1}^\infty \subset {\mathcal M}_{t,\rho}({\mathbb R}^n)$
and
$\{\lambda_j\}_{j=1}^\infty \subset[0,\infty)$
fulfill
\[
\|a_j\|_{{\mathcal M}_{t,\rho}} \le |Q_j|^{\frac{n-\rho}{nt}}, \quad
{\rm supp}(a_j) \subset Q_j, \quad
\left\|\sum_{j=1}^\infty \lambda_j \chi_{Q_j}
\right\|_{{\mathcal M}_{q,\lambda}}
<\infty.
\]
Then
$f \equiv \sum_{j=1}^\infty \lambda_j a_j$
converges in
${\mathcal S}'({\mathbb R}^n) \cap L^q_{\rm loc}({\mathbb R}^n)$
and satisfies
\begin{equation}\label{eq:thm1-2a}
\|f\|_{{\mathcal M}_{q,\lambda}}
\lesssim
\left\|\sum_{j=1}^\infty\lambda_j\chi_{Q_j}\right\|_{{\mathcal M}_{q,\lambda}}.
\end{equation}
An example in \cite[Section 4]{SaSuTa11-1}
shows that we can not let $q=r$.
Meanwhile, when $q=1$, Theorem \ref{lem:131108-2} shows
that we can take $r=1$.
\end{remark}

\subsection{Boundedness of the maximal operator}

Below we write $W{\mathcal M}_{1,\lambda}({\mathbb R}^n)$
to denote the weak Morrey space;
a measurable function $f$ belongs to
$W{\mathcal M}_{1,\lambda}({\mathbb R}^n)$
if and only if
\[
\|f\|_{W{\mathcal M}_{1,\lambda}}
\equiv
\sup_{T>0}T\|\chi_{\{|f|>T\}}\|_{{\mathcal M}_{1,\lambda}}
<\infty.
\]
The following result is standard and we aim
to extend it to generalized Morrey spaces.
\begin{theorem}{\rm \cite{CF}}\label{ChirFr}
Let $0<\lambda <n$. Then{\rm;}
\begin{enumerate}
\item[$(1)$]
$M$ is bounded on ${\mathcal M}_{p,\lambda}(\mathbb{R}^n)$ if $1<p<\infty$;
\item[$(2)$]
$M$ is bounded from ${\mathcal M}_{1,\lambda}(\mathbb{R}^n)$
to $W{\mathcal M}_{1,\lambda}(\mathbb{R}^n)$.
\end{enumerate}
\end{theorem}

We denote by $L_{\infty,v}(0,\infty)$ the space of all
functions $g(t)$, $t>0$ such that
$$
\|g\|_{L_{\infty,v}(0,\infty)} \equiv \sup_{t>0}v(t)g(t)
$$ is finite
and $L_{\infty}(0,\infty) \equiv L_{\infty,1}(0,\infty)$.
The space ${\mathfrak M}(0,\i)$ is defined to be the set of all Lebesgue-measurable
functions on $(0,\i)$ and ${\mathfrak M}^+(0,\i)$ its subset
consisting of all nonnegative functions on $(0,\i)$. We denote by
${\mathfrak M}^+\!(0,\i;\!\uparrow\!)\!$ the cone of all functions in
${\mathfrak M}^+(0,\i)$ which are non-decreasing on $(0,\i)$ and
$$
\A\equiv\left\{\phi \in {\mathfrak M}^+(0,\i;\uparrow):
\lim_{t\rightarrow 0+}\phi(t)=0\right\}.
$$
Let $u$ be a continuous and non-negative function on $(0,\i)$. We
define the supremal operator
$\overline{S}_{u}$ on $g\in {\mathfrak M}(0,\i)$ by
$$
(\overline{S}_{u}g)(t)\equiv \|u\, g\|_{L_{\i}(t,\i)},~~t\in (0,\i).
$$
We invoke the following theorem.
\begin{theorem}{\rm \cite{BurGogGulMus1}}\label{thm5.1}
Let $v_1$, $v_2$ be non-negative measurable functions satisfying
$0<\|v_1\|_{L_{\i}(t,\i)}<\i$ for any $t>0$ and let $u$ be a continuous
non-negative function on $(0,\i)$

Then the operator $\overline{S}_{u}$ is bounded from
$L_{\i,v_1}(0,\i)$ to $L_{\i,v_2}(0,\i)$ on the cone $\A$ if and
only if
\begin{equation}\label{eq6.8}
\begin{split}
\left\|v_2 \overline{S}_{u}\left( \| v_1 \|^{-1}_{L_{\i}(\cdot,\i)}
\right)\right\|_{L_{\i}(0,\i)}<\i.
\end{split}
\end{equation}
\end{theorem}

\

We will use the following statement on the boundedness
of the weighted Hardy operator
$$
H^{\ast}_{w} g(t)\equiv \int_t^{\infty} g(s) w(s) ds,~ \ \ 0<t<\infty,
$$
where $w$ is a weight.

The following theorem in the case $w=1$ was proved
in \cite[Theorem 5.1]{BurGogGulMus1}.

\begin{theorem}{\rm \cite[Theorem 3.1]{GulJMS2013}}\label{thm3.2.}
Let $v_1$, $v_2$ and $w$ be weights on $(0,\infty)$ and
assume that $v_1$ is bounded outside a neighborhood of the origin. 
The inequality
\begin{equation} \label{vav01}
\sup _{t>0} v_2(t) H^{\ast}_{w} g(t) \leq C \sup _{t>0} v_1(t) g(t)
\end{equation}
holds for some $C>0$ for all non-negative and non-decreasing $g$ on $(0,\i)$ if and
only if
\begin{equation} \label{vav02}
B\equiv \sup _{t>0} v_2(t)\int_t^{\infty} \frac{w(s) ds}{\sup _{s<\tau<\infty} v_1(\tau)}<\infty.
\end{equation}
Moreover, the value $C=B$ is the best constant for \eqref{vav01}.
\end{theorem}

\begin{remark}\label{rem2.3.}
In \eqref{vav01} and \eqref{vav02} it will be understood
that $\frac{1}{\i}\equiv 0$ and $0 \cdot \i\equiv 0$.
See \cite[Theroem 1]{GulAJM2013} as well for some application.
\end{remark}

The following statement, extending the results
in T.~Mizuhara and E. Nakai \cite{Mi, Nakai94},
was proved in V.S. Guliyev \cite{GulDoc};
see also \cite{GulBook, GulJIA}.
\begin{proposition} \label{nakaiVagif0}
Let $1 \le p < \infty$. Moreover,
let $\phi_1$, $\phi_2$ be positive measurable functions satisfying
\begin{equation} \label{MizN}
\int_t^{\infty} \phi_1(x,\tau)\frac{d\tau}{\tau} \lesssim \phi_2(x,t)
\end{equation}
for all $t>0$. Then, for $p>1$, $M$ is bounded from ${\mathcal M}_{p,\phi_1}({\mathbb R}^n)$ to ${\mathcal M}_{p,\phi_2}({\mathbb R}^n)$ and, for $p=1$, $M$ is bounded
from ${\mathcal M}_{1,\phi_1}({\mathbb R}^n)$ to $W{\mathcal M}_{1,\phi_2}({\mathbb R}^n)$.
\end{proposition}

The following statements, containing Proposition \ref{nakaiVagif0}, 
was proved by A. Akbulut, V.S. Guliyev and R. Mustafayev \cite{AkbGulMus1};
note that (\ref{MizN}) is stronger than (\ref{eq3.6.VZ}).
\begin{proposition} \label{3.4.}
Let $1\le p<\infty$ and suppose the couple $(\phi_1,\phi_2)$ satisfies the following condition{\rm;}
\begin{equation}\label{eq3.6.VZ}
\sup\limits_{\tau>t}
\left(
\inf\limits_{\tau<s<\infty}\phi_1(x,s) s^{\frac{n}{p}}\tau^{-\frac{n}{p}}
\right) \lesssim \phi_2(x,t),
\end{equation}
where the impicit constant does not depend on $x$ and $t$.
Then, for $p>1$, $M$ is bounded from ${\mathcal M}_{p,\phi_1}({\mathbb R}^n)$ to ${\mathcal M}_{p,\phi_2}({\mathbb R}^n)$ and, for $p=1$,
$M$ is bounded from ${\mathcal M}_{1,\phi_1}({\mathbb R}^n)$ to $W{\mathcal M}_{1,\phi_2}({\mathbb R}^n)$.
Namely, for $p>1$,
\begin{equation*}
\|Mf\|_{{\mathcal M}_{p,\phi_2}} \lesssim \|f\|_{{\mathcal M}_{p,\phi_1}}
\end{equation*}
for all $f \in {\mathcal M}_{p,\phi_1}({\mathbb R}^n)$ and for $1\le p<\infty$,
\begin{equation*}
\|Mf\|_{W{\mathcal M}_{p,\phi_2}} \lesssim \|f\|_{{\mathcal M}_{p,\phi_1}}
\end{equation*}
for all $f \in {\mathcal M}_{p,\phi_1}({\mathbb R}^n)$.
\end{proposition}

From this proposition,
when $\phi_1=\phi_2=\phi$,
we have the following boundedness.
We know that there is no requirement
when we consider the boundedness
of the maximal operator \cite[Theorem 2.3]{Sa08-1}
when $\phi$ is independent of $x$.
Proposition \ref{3.4.} naturally extends the assertion above.
\begin{corollary}{\rm \cite[Theorem 1]{Nakai94}, \cite[Theorem 2.3]{Sa08-1}} \label{jena01}
Let $1 \le p<\infty$ and $\phi \in {\mathcal G}_p$.
\begin{enumerate}
\item
Let $1< p<\infty$.
Then
$\|Mf\|_{{\mathcal M}_{p,\phi}}\lesssim \|f\|_{{\mathcal M}_{p,\phi}}$
for all $f \in {\mathcal M}_{p,\phi}({\mathbb R}^n)$.
\item
Let $1\le p<\infty$.
Then
$\|Mf\|_{W{\mathcal M}_{p,\phi}} \lesssim \|f\|_{{\mathcal M}_{p,\phi}}$
for all $f \in {\mathcal M}_{p,\phi}({\mathbb R}^n)$.
\end{enumerate}
\end{corollary}

By using the Planchrel-P\'{o}lya Nilokiski'i inequality \cite{PlPo37},
we have the following estimate of the Peetre maximal operator.

\begin{theorem}\label{Trburgul1}
Let $0 < p < \infty $, $0 < r \le p$ and
suppose that the couple $(\phi_1,\phi_2)$ satisfies the condition 
\begin{equation}\label{eq3.6.VZM}
\sup\limits_{\tau>t}\left[
\left(
\inf\limits_{\tau<s<\infty}
\phi_1(x,s) s^{\frac{nr}{p}}
\right)\tau^{-\frac{nr}{p}}\right]
\lesssim
\phi_2(x,t),
\end{equation}
where the implicit constant does not depend on $x$ and $t$.
Let $\Omega $ be a compact set, $d$ be the diameter of $\Omega$.
\begin{enumerate}
\item[$(1)$]
If $r < p < \infty $, $f$ belongs to ${\mathcal M}_{p,\phi_1}({\mathbb R}^n)$
and $\supp{\mathcal F}f \subset \Omega$, then
$$
 \left\|\sup_{y \in \Rn}\frac{|f(\cdot-y)|}{1+ d |y|^{n/r}}
 \right\|_{{\mathcal M}_{p,\phi_2}} \lesssim \| f \|_{{\mathcal M}_{p,\phi_1}}.
$$
\item[$(2)$]
If $p = r$, $f$ belongs to ${\mathcal M}_{p,\phi_1}({\mathbb R}^n)$
and $\supp{\mathcal F}f \subset \Omega$, then
$$
 \left\|\sup_{y \in \Rn}\frac{|f(\cdot-y)|}{1+ d |y|^{n/r}}
 \right\|_{W{\mathcal M}_{r,\phi_2}} \lesssim \| f \|_{{\mathcal M}_{r,\phi_1}}.
$$
\end{enumerate}
\end{theorem}

\begin{proof}
We have
\begin{equation}\label{ineq07}
\frac{\left| g(x-y) \right|}{1+|y|^{n/r}} \lesssim
\left[M \left(|g|^r\right)(x) \right]^{1/r}
\end{equation}
for all $x, \, y \in \Rn$ and $g \in {\mathcal S}'(\Rn)$
such that ${\mathcal F}g$ is supported on $\Omega$;
see \cite[p. 22]{Trieb1}.
By \eqref{ineq07} and Proposition \ref{3.4.},
we conclude Theorem \ref{Trburgul1}.
\end{proof}

\begin{remark}
Theorem \ref{Trburgul1} is proved in \cite{SaTa1,TaXu}
in the case of classical Morrey spaces.
\end{remark}

For $p\in[1,\infty)$,
we have a counterpart to Corollary \ref{cor:150308-1}:
\begin{lemma}
Let $1<p\le \infty$ and $\phi \in {\mathcal G}_p$
Then for all $\kappa \in {\mathcal S}({\mathbb R}^n)$,
\begin{equation}\label{eq:131211-1}
\int_{{\mathbb R}^n}|\kappa(x)f(x)|\,dx
\lesssim
\|f\|_{{\mathcal M}_{p,\phi}}
\sup_{x \in {\mathbb R}^n}(1+|x|)^{2n+1}|\kappa(x)|.
\end{equation}
\end{lemma}

\begin{proof}
We decompose the left-hand side as follows:
\begin{align*}
&\int_{\R^n} |\kappa(x)f(x)| \ dx \\
&\leq \int_{[-1,1]^n} |\kappa(x)f(x)| \ dx
+\sum\limits_{j=1}^{\infty} \int_{[-(j+1),(j+1)]^n \setminus [-j,j]^n} |\kappa(x)f(x)| \ dx\\
&\leq
\left(\sup\limits_{x\in [-1,1]^n} |\kappa(x)| \right)
\|f\|_{L_1([-1,1]^n)}
+\sum\limits_{j=1}^{\infty} \int\limits_{[-(j+1),(j+1)]^n \setminus [-j,j]^n}
\frac{\left|x^{2n+1}\kappa(x) \right| |f(x)|}{j^{2n+1}} \ dx\\
&\leq \left(\sup_{x \in {\mathbb R}^n}(1+|x|)^{2n+1}|\kappa(x)|\right)
\left(\|f\|_{L^1([-1,1]^n)}+ \sum\limits_{j=1}^{\infty}
\int\limits_{[-(j+1), (j+1)]^n \setminus [-j,j]^n}
\frac{|f(x)|}{j^{2n+1}} \ dx\right).
\end{align*}
By the definition of the maximal operator, for all $j \in {\mathbb N}$, we have
\[
\frac{1}{|[-j,j]^n|}\int_{[-j,j]^n}|f(y)|\,dy
\le
Mf(x), \quad x \in [-1,1]^n.
\]
Thus,
\[
 \frac{1}{\phi({\rm o},2)2^{n/p}}
\left\|\frac{1}{|[-j,j]^n|}\int_{[-j,j]^n}|f(y)|\,dy
\right\|_{L^p([-1,1]^n)}
\le
\|Mf\|_{{\mathcal M}_{p,\phi}}
\lesssim
\|f\|_{{\mathcal M}_{p,\phi}}.
\]
This means
$\displaystyle
\int_{[-j,j]^n}|f(y)|\,dy
\lesssim j^n\|f\|_{{\mathcal M}_{p,\phi}}.
$
Thus,
\begin{align*}
\int_{\R^n} |\kappa(x)f(x)| dx &\lesssim \|f\|_{{\mathcal M}_{p,\phi}} \sup\limits_{x\in \R^n} (1+|x|)^{2n+1} |\kappa(x)|
\left( 1+\sum\limits_{j=1}^{\infty} \frac{(j+1)^n}{j^{2n+1}}\right)\\
&\sim \|f\|_{{\mathcal M}_{p,\phi}} \sup\limits_{x\in \R^n} (1+|x|)^{2n+1} |\kappa(x)|,
\end{align*}
as was to be shown.
\end{proof}

\subsection{Vector-valued boundedness of the maximal operator}

Our aim here
is to extend the Fefferman-Stein vector-valued inequality
to our function spaces
for $M$ in addition to Corollary \ref{burgul1lq};
\begin{equation}\label{eq:FS}
\left\|\left(\sum_{j=1}^\infty Mf_j{}^u\right)^{\frac1u}\right\|_{L^p}
\lesssim
\left\|\left(\sum_{j=1}^\infty |f_j|^u\right)^{\frac1u}\right\|_{L^p},
\end{equation}
where $1<p<\infty$ and $1<u \le \infty$;
see \cite{FeffSt} for the proof of (\ref{eq:FS}).
When $q=\infty$, it is understood that (\ref{eq:FS}) reads;
\[
\left\|\sup_{j \in {\mathbb N}} Mf_j\right\|_{L^p}
\lesssim
\left\|\sup_{j \in {\mathbb N}} |f_j|\right\|_{L^p}.
\]


Our main result here is as follows:
\begin{theorem} \label{3.4.VectV}
Let $1\le q \le \infty$ and suppose that
the couple $(\phi_1,\phi_2)$ satisfies the condition{\rm;}
\begin{equation}\label{eq3.6.VZInt}
\int_t^{\i}
\left(\inf\limits_{\tau<s<\infty}\phi_1(x,s)s^{\frac{n}{p}}\right)
\frac{d\tau}{\tau^{\frac{n}{p}+1}}
\lesssim \phi_2(x,t),
\end{equation}
where the implicit constant does not depend on $x$ and $t$.
\begin{enumerate}
\item[$(1)$]
For $1< p<\infty$,
$M$ is bounded from ${\mathcal M}_{p,\phi_1}(l_{q},{\mathbb R}^n)$
to ${\mathcal M}_{p,\phi_2}(l_{q},{\mathbb R}^n)$,
i.e.,
\begin{equation*}
\|MF\|_{{\mathcal M}_{p,\phi_2}(l_{q})} \lesssim \|F\|_{{\mathcal M}_{p,\phi_1}(l_{q})}
\end{equation*}
holds for all $F \in {\mathcal M}_{p,\phi_1}(l_{q},{\mathbb R}^n)$.
\item[$(2)$]
For $1\le p<\infty$,
$M$ is bounded from ${\mathcal M}_{p,\phi_1}(l_{q},{\mathbb R}^n)$
to $W{\mathcal M}_{p,\phi_2}(l_{q},{\mathbb R}^n)$,
i.e.,
\begin{equation*}
\|MF\|_{W{\mathcal M}_{1,\phi_2}(l_{q})} \lesssim \|F\|_{{\mathcal M}_{1,\phi_1}(l_{q})}
\end{equation*}
holds for all $F \in {\mathcal M}_{1,\phi_1}(l_{q},{\mathbb R}^n)$.
\end{enumerate}
\end{theorem}

As a corollary, we can recover the vector-valued inequality
obtained in \cite[Theorem 5.3]{GHS-pre}.
\begin{corollary}{\rm \cite[Theorem 5.3]{GHS-pre}}
\label{cor:140327-111}
Let $1<p<\infty$ and $1<u \le \infty$.
Assume in addition that $\phi \in {\cG}_p$ satisfies
$(\ref{eq:140327-3})$.
Then
\[
\left\|\left(\sum_{j=1}^\infty Mf_j{}^u
\right)^{\frac1u}\right\|_{{\mathcal M}_{p,\phi}}
\lesssim
\left\|\left(\sum_{j=1}^\infty |f_j|^u
\right)^{\frac1u}\right\|_{{\mathcal M}_{p,\phi}}
\]
for any sequence of measurable functions
$\{f_j\}_{j=1}^\infty$.
\end{corollary}

The proof of Theorem \ref{3.4.VectV} is postponed
till the latter half of this section.
We start with a direct corollary of Proposition \ref{3.4.},
which can be readily extended to the following vector-valued case.
Write $MF\equiv \{ Mf_j \}_{j=-\infty}^{\infty}$,
when we are given a sequence $F=\{ f_j\}_{j=-\infty}^{\infty}$.

\begin{corollary}\label{burgul1lq}
Let $1 \le q \le \infty $ and
suppose that the couple
$(\phi_1,\phi_2)$ satisfies the condition \eqref{eq3.6.VZ}.
\begin{enumerate}
\item[$(1)$]
Let $1< p<\infty$.
Then $M$ is bounded from
$l_q\left({\mathcal M}_{p,\phi_1},{\mathbb R}^n\right)$
to
$l_q\left({\mathcal M}_{p,\phi_2},{\mathbb R}^n\right)$.
That is,
$$
\|MF\|_{l_q\left({\mathcal M}_{p,\phi}\right)}
\lesssim
\|F\|_{l_q\left({\mathcal M}_{p,\phi}\right)}
$$
for all sequences of measurable functions $F=\{f_j\}_{j=1}^\infty$.
\item[$(2)$]
Let $1\le p<\infty$.
Then
$M$ is bounded from
$l_q\left({\mathcal M}_{1,\phi_1},{\mathbb R}^n\right)$
to
$l_q\left(W{\mathcal M}_{1,\phi_2},{\mathbb R}^n\right)$.
That is,
$$
\|MF\|_{l_q\left(W{\mathcal M}_{1,\phi}\right)}
\lesssim
\|F\|_{l_q\left({\mathcal M}_{1,\phi}\right)}
$$
for all sequences of measurable functions $F=\{f_j\}_{j=1}^\infty$.
\end{enumerate}
\end{corollary}

Now we are oriented to the proof of Theorem \ref{3.4.VectV}.
We first prove the following auxiliary estimate:
\begin{lemma}\label{lem4.1Vect}
Let $1<p < \infty$ and $1< q \le \infty$.
\begin{enumerate}
\item
Then the inequality
\begin{gather}\label{eq100fgd}
\|\;\|MF\|_{l_q}\|_{L^p(B(x,r))}
\lesssim
\|\;\|F\|_{l_q}\|_{L^p(B(x,2r))} + r^{\frac{n}{p}}
\int_r^\infty
\frac{\|\;\|F\|_{l_q}\|_{L_1(B(x,t))}}{t^{n+1}}\,dt
\end{gather}
holds for all $F=\{ f_j\}_{j=0}^{\infty} \subset \Lploc$
and for any ball $B=B(x,r)$ in $\Rn$.
\item
Moreover, the inequality
\begin{gather}\label{eq100WZxc}
\|\;\|MF\|_{l_q}\|_{WL_1(B(x,r))}
\lesssim
\|\;\|F\|_{l_q}\|_{L_1(B(x,2r))} + r^{n}
\int_r^\infty
\frac{\|\;\|F\|_{l_q}\|_{L_1(B(x,t))}}{t^{n+1}}\,dt
\end{gather}
holds for all $F=\{ f_j\}_{j=0}^{\infty} \subset \Lloc$
and for any ball $B=B(x,r)$ in $\Rn$.
\end{enumerate}
\end{lemma}

\begin{proof}
Let $1< p<\i$ and $1 \le q \le \infty$.
We split $F=\{f_j\}_{j=-\i}^{\i}$ with
\begin{align}\label{repr}
&F=F_1+F_2, ~~~~
F_1=\{ f_{j,1}\}_{j=-\infty}^{\infty}, ~~~~
F_2=\{ f_{j,2}\}_{j=-\i}^{\i},
\\
&f_{j,1}(y)=f_j(y)\chi_{B(x,3r)}(y),\quad
 f_{j,2}(y)=f_j(y)\chi_{\Rn\backslash B(x,3r)}(y),
\quad r>0. \notag
\end{align}
It is obvious that
\begin{align*}
\|\;\|MF\|_{l_q}\|_{L^p(B(x,r))}
\leq
\|\;\|MF_1\|_{l_q}\|_{L^p(B(x,r))}
+
\|\;\|MF_2\|_{l_q}\|_{L^p(B(x,r))}
\end{align*}
for any ball $B=B(x,r)$.

At first,
we shall estimate
$\|\;\|MF_1\|_{l_q}\|_{L^p(B(x,r))}$
for $1<p<\infty$.
Thanks to the well-known Fefferman-Stein maximal inequality,
we have
\begin{align} \label{GAzel1}
\|\;\|MF_1\|_{l_q}\|_{L^p(B(x,r))} & \le
\|\;\|MF_1\|_{l_q}\|_{L^p(\Rn)} \notag
\\
& \lesssim \|\;\|F_1\|_{l_q}\|_{L^p(\Rn)}\\
&=\|\;\|F\|_{l_q}\|_{L^p(B(x,2r))},
\end{align}
where the implicit constant is independent of the vector-valued function $F$.
Thus, the estimate for $MF_1$ is valid.

Now we handle $MF_2$.
Freeze a point $y$ in $B(x,r)$.

We begin with two geometric observations.
First if $B(y,t)\cap (\Rn\backslash B(x,3r))\neq\emptyset,$
then $t>r$.
Indeed, if $z\in B(y,t)\cap (\Rn\backslash B(x,3r)),$
then $$t > |y-z| \geq |x-z|-|x-y|>2r-r=r.$$
Next,
$B(y,t)\cap (\Rn\backslash B(x,3r))\subset B(x,2t)$.
Indeed, for $z\in B(y,t)\cap (\Rn\backslash B(x,3r))$,
then we get $|x-z|\leq |y-z|+|x-y|<t+r<2t$.

Hence for all $j \in \Nn$
\begin{equation*}
\begin{split}
\|MF_2(y)\|_{l_q} & = \Big\|\sup_{t>0}\frac{1}{|B(y,t)|} \int_{B(y,t)\cap {(\Rn\backslash B(x,2r))}}|f_{j}(z)|d z\Big\|_{l_q}
\\
&\lesssim
\left\|
\sum\limits_{k=1}^{\infty}
\frac{1}{|2^k Q|}\int_{2^k Q}|f_j(y)| dy
\right\|_{l_q}
\\
&\lesssim
\sum\limits_{k=1}^{\infty}
\frac{1}{|2^k Q|}\int_{2^k Q}\left\|F(y)\right\|_{l_q} dy.
\end{split}
\end{equation*}
As a result, we obtain
\begin{equation}\label{eq:140327-111}
\|MF_2(y)\|_{l_q}
\lesssim
\int_{2r}^\infty s^{-n-1}\|\,\|F\|_{l_q}\|_{B(x,s)}\,ds
\end{equation}
for all $y \in B(x,s)$.
Thus, the estimate for $MF_2$ is valid.

Then obtain \eqref{eq100fgd} from \eqref{GAzel1}
and (\ref{eq:140327-111}).

Let $p=1$. It is obvious that for any ball $B=B(x,r)$
\begin{align*}
\|\;\|MF\|_{l_q}\|_{WL_1(B(x,r))}
\leq
\|\;\|MF_1\|_{l_q}\|_{WL_1(B(x,r))}
+
\|\;\|MF_2\|_{l_q}\|_{WL_1(B(x,r))}.
\end{align*}

By the weak-type Fefferman-Stein maximal inequality
(see \cite{FeffSt}) we have
\begin{align} \label{GAzel1W}
\|\;\|MF_1\|_{l_q}\|_{WL_1(B(x,r))} & \le
\|\;\|MF_1\|_{l_q}\|_{WL_1(\Rn)} \notag
\\
& \lesssim \|\;\|F_1\|_{l_q}\|_{L_1(\Rn)}=\|\;\|F\|_{l_q}\|_{L_1(B(x,2r))},
\end{align}
where the implicit constant is independent of the vector-valued function $F$.

Then by \eqref{GAzel1W} and (\ref{eq:140327-111}),
we obtain the inequality \eqref{eq100WZxc}.
\end{proof}

We transform Lemma \ref{lem4.1Vect} to an inequality
we use to prove Theorem \ref{3.4.VectV}.

\begin{lemma}\label{lem4.1DVect}
Let $1 \le p < \infty$ and $1 \le q \le \infty$.
Then, for any ball $B=B(x,r)$ in $\Rn$, the inequality
\begin{gather}\label{eq100Dx}
\|\;\|MF\|_{l_q}\|_{L^p(B(x,r))} \lesssim
r^{\frac{n}{p}} \int_{2r}^\infty
t^{-\frac{n}{p}-1} \|\;\|F\|_{l_q}\|_{L^p(B(x,t))}\,dt
\end{gather}
holds for all $F=\{ f_j\}_{j=-\infty}^{\infty}\subset \Lploc$.

Moreover, for any ball $B=B(x,r)$ in $\Rn$, the inequality
\begin{gather}\label{eq100WZDx}
\|\;\|MF\|_{l_q}\|_{WL_1(B(x,r))} \lesssim
r^{n} \int_{2r}^\infty t^{-n-1} \|\;\|F\|_{l_q}\|_{L_1(B(x,t))}\,dt
\end{gather}
holds for all $F=\{ f_j\}_{j=-\infty}^{\infty} \subset \Lloc$.
\end{lemma}

\begin{proof}
Note that, for all $1 \le p < \i$
\begin{align} \label{mzsa}
\|\;\|F\|_{l_q}\|_{L^p(B(x,2r))}
\le r^{\frac{n}{p}} \
\int_{2t}^\infty t^{-\frac{n}{p}-1} \|\;\|F\|_{l_q}\|_{L^p(B(x,t))}\,dt.
\end{align}

Applying H\"older's inequality, we get
\begin{align}\label{Ga13dfkk}
\|\;\|MF_2(y)\|_{l_q}\|_{L_{p}(B(x,r))} & \le 2^{n} v_n^{-1} \, \|1\|_{L^p(B(x,r))}\,
\int_{2r}^\infty t^{-n-1} \,
\|\;\|F\|_{l_q}\|_{L^1(B(x,t))}\,dt \notag
\\
& \le 2^{n} v_n^{-1} \, \|1\|_{L^p(B(x,r))}\,
\int_{2r}^\infty t^{-n/p-1} \,
\|\;\|F\|_{l_q}\|_{L^p(B(x,t))}\,dt.
\end{align}
Since $\|1\|_{L^p(B(x,r))} = v_n^{\frac{1}{p}} r^{\frac{n}{p}}$,
we then obtain \eqref{eq100Dx} from \eqref{mzsa} and \eqref{Ga13dfkk}.

Let $p=1$.
The inequality \eqref{eq100WZDx} directly
follows from \eqref{mzsa}.
\end{proof}

With these estimates in mind,
let us prove Theorem \ref{3.4.VectV}.
\begin{proof}[Proof of Theorem \ref{3.4.VectV}]
By Lemma \ref{lem4.1DVect} and Theorem \ref{thm3.2.} with
$v_1(r)=\phi_1(x,r)^{-1}r^{-\frac{n}{p}}$,
$v_2(r)=\phi_2(x,r)^{-1}$,
$g(r)=\|\;\|F\|_{l_q}\|_{L^p(B(x,r))}$
and $w(r)=r^{-\frac{n}{p}-1}$, we have
\begin{equation*}
\begin{split}
\|MF\|_{{\mathcal M}_{p,\phi_2}(l_{q})} & \lesssim \sup_{x\in\Rn,\,r>0}\phi_2(x,r)^{-1} \int_{2r}^\infty
t^{-\frac{n}{p}-1} \|\;\|F\|_{l_q}\|_{L^p(B(x,t))}\,dt
\\
& \lesssim \sup_{x\in\Rn,\,r>0} \phi_1(x,r)^{-1}\,r^{-\frac{n}{p}} \|\;\|F\|_{l_q}\|_{L^p(B(x,r))}
\\
& =\|F\|_{{\mathcal M}_{p,\phi_1}(l_{q})},
\end{split}
\end{equation*}
if $1<p<\i$ and
\begin{equation*}
\begin{split}
\|MF\|_{W{\mathcal M}_{1,\phi_2}(l_{q})} & \lesssim \sup_{x\in\Rn,\,r>0}\phi_2(x,r)^{-1} \int_{2r}^\infty
t^{-n-1} \|\;\|F\|_{l_q}\|_{L^p(B(x,t))}\,dt
\\
& \lesssim \sup_{x\in\Rn,\,r>0} \phi_1(x,r)^{-1}\,r^{-n} \|\;\|F\|_{l_q}\|_{L^p(B(x,r))}
\\
& =\|F\|_{{\mathcal M}_{1,\phi_1}(l_{q})}
\end{split}
\end{equation*}
if $p=1$.
\end{proof}

As a corollary, we can recover the result in \cite{Sa08-1}.
\begin{corollary}{\rm \cite[Theorem 2.5]{Sa08-1}}\label{jena03}
Let $1\le p<\infty$, $1 \le q \le \infty $
 and $\phi \in {\mathcal G}_p$.
Assume in addition
\[
\int_r^\infty \phi(t)\frac{dt}{t} \lesssim \phi(r).
\]
Then
\begin{equation}\label{eq:140327-11111}
\|MF\|_{{\mathcal M}_{p,\phi}(l_{q})}
\lesssim
\|F\|_{{\mathcal M}_{p,\phi}(l_{q})} ~~~ \mbox{if} ~~~ p>1,
\end{equation}
and
$$
\|MF\|_{W{\mathcal M}_{p,\phi}(l_{q})}
\lesssim
\|F\|_{{\mathcal M}_{1,\phi}(l_{q})}.
$$
\end{corollary}

We can prove the following estimate,
which is a counterpart to Theorem \ref{Trburgul1}.
\begin{theorem}\label{tangXu}
Let $0 < p < \infty $, $0 < q \le \infty $,
$0 < r < \min \{p, q \}$ and
suppose that the couple $(\phi_1,\phi_2)$
satisfies the condition \eqref{eq3.6.VZM}. Let
$\{\Omega_j \}_{j \in \N_{0}}$ be a sequence of
compact sets, and let
$d_j$ be the diameter of $\Omega_j$.
Then exists a positive constant $C$ such that
\begin{equation}\label{eq:140327-100}
 \left\|
\left\{\left\|\sup_{y \in \Rn}\frac{|f_j(\cdot -y)|}{1+ d_j |y|^{n/r}}
 \right\|_{{\mathcal M}_{p,\phi_2}} \right\}_{j \in \N_0}\right\|_{l_q}
 \le C \left\| \left\{ \left\|f_j\right\|_{{\mathcal M}_{p,\phi_1}} \right\}_{j \in \N_0} \right\|_{l_q},
\end{equation}
if we are given a collection
of measurable functions $\{f_j\}_{j \in \Nn}$
such that ${\rm supp}({\cal F}f_j) \subset \Omega_j$.
\end{theorem}

\begin{proof}
We begin with a reduction.
Let $\{f_j \}_{j \in \N_{0}} \in l_q\left({\mathcal M}_{p,\phi_1}\right)$,
$0 < q \le \infty$, $\eta^j \in \Omega$
and $h_j(x)\equiv e^{-i x \eta^j} f_j(x)$.
We have ${\cal F}h_j(\xi)={\cal F}f_j(\xi+\eta^j)$.
Therefore,
$ \supp {\cal F}h_j \subset \Omega_j - \eta^j$.

If $ \{f_j \}_{j \in \N_{0}} \in l_q\left({\mathcal M}_{p,\phi_1,\Omega}\right)$,
$\eta^j \in \Omega$
satisfies $(1)$ of Theorem \ref{tangXu},
then so does $\{h_j\}_{j \in \N_{0}}$ also, where $\Omega$ is replaced by
$\{\Omega_j - \eta^j\}_{j \in \N_{0}}$, and the converse also
holds. Thus, we may suppose $0 \in \Omega_j$.
It suffices to consider the case $\Omega_j = D_j = B(0,d_j)$.

Second, we have to prove $(\ref{eq:140327-100})$,
 when $\Omega_j = D_j =
B(0,d_j)$ and $d_j>0$. If $ \{f_j \}_{j \in \N_{0}} \in
l_q\left({\mathcal M}_{p,\phi_1,\Omega}\right)$, $\eta^j \in \Omega$,
then $f_j \in L_{p,\Omega_j}$. If $g_j\equiv f_j(d^{-1}\cdot)$, then
${\cal F} g_j=d_j^n ({\cal F}f_j)(d_j\cdot)$ and
hence $\supp {\cal F}g_j \subset B(0,1)$.

From \eqref{ineq07}, we obtain
\begin{equation}\label{ineq17}
\frac{\left| f_j(x-z) \right|}{1+|d_j z|^{n/r}} \lesssim
\left[M\left( |f_j|^r\right) \right]^{1/r}, \quad\mbox{for all} \quad x,
\, z \in \Rn,
\end{equation}
where the implicit constant is independent of $x, \, z$ and $j$.

If $0 < q < \infty$, then by \eqref{ineq17} we have
\begin{align*}
\left\| \left\{ \sup_{x \in \Rn}\frac{\left| f_j(\cdot-z) \right|}{1+|d_j
z|^{n/r}} \right\}_{j \in \N_0} \right\|_{l_q\left({\mathcal M}_{p,\phi_2}\right)}
& \lesssim\left\| \left\{ \left[ M \left( |f_j|^r\right)
\right]^{1/r}\right\}_{j \in \N_0} \right\|_{l_q\left({\mathcal M}_{p,\phi_2}\right)}
\\
&=\left\| \left\{ M \left( |f_j|^r\right) \right\}_{j \in \N_0}
\right\|_{l_{q/r}\left(M_{p/r,\phi_2}\right)}^{1/r}.
\end{align*}
Since $0 < r < \min\{ p, q \}$, we have $p/r > 1$ and
$q/r > 1$. By Corollary \ref{burgul1lq}, we have
\begin{align*}
\left\|\left\{ \sup_{x \in \Rn}\frac{\left| f_j(\cdot-z) \right|}{1+|d_j
z|^{n/r}} \right\}_{j \in \N_0} \right\|_{l_q\left({\mathcal M}_{p,\phi_2}\right)}
& \lesssim \left\| \left\{ M \left( |f_j|^r\right) \right\}_{j \in \N_0}
\right\|_{l_{q/r}\left(M_{p/r,\phi_2}\right)}^{1/r}
\\
& \lesssim \left\| \{f_j \}_{j \in \N_0}
\right\|_{l_q\left({\mathcal M}_{p,\phi_1}\right)}.
\end{align*}

If $q=\infty$, by \eqref{ineq17}, we have
\begin{equation*}
\sup\limits_{j \in {\mathbb N}_0}
\sup_{z \in \Rn}\frac{\left| f_j(\cdot-z)
\right|}{1+|d_j z|^{n/r}}
\lesssim
\sup\limits_{j \in {\mathbb N}_0} \left[M \left(|f_j|^r\right)(x)\right]^{1/r}.
\end{equation*}
Thus,
\begin{align*}
\sup\limits_{j \in {\mathbb N}_0} \left\| 
\sup_{z \in \Rn}\frac{\left| f_j(\cdot-z)
\right|}{1+|d_j z|^{n/r}} \right\|_{{\mathcal M}_{p,\phi_2}}
& \lesssim \sup\limits_{j \in {\mathbb N}_0} \left\| 
]\left[M \left(|f_j|^r\right)(\cdot)
\right]^{1/r} \right\|_{{\mathcal M}_{p,\phi_2}}
\\
& = \sup\limits_{j \in {\mathbb N}_0} \left\| M \left(|f_j|^r\right)
\right\|_{M_{p/r,\phi_2}}^{1/r}.
\end{align*}
Using Corollary \ref{burgul1lq}, we obtain
\begin{equation*}
\sup\limits_{j \in {\mathbb N}_0}
\left\|\sup_{z \in \Rn}\frac{\left| f_j(\cdot-z)
\right|}{1+|d_j z|^{n/r}}\right\|_{{\mathcal M}_{p,\phi_2}} 
\lesssim
\sup\limits_{j \in {\mathbb N}_0} 
\left\|f_j \right\|_{{\mathcal M}_{p,\phi_1}}.
\end{equation*}
The theorem is therefore proved.
\end{proof}

Finally, to conclude this section, we obtain
an estimate which is used later.

\begin{lemma}\label{lem:140327-1}
Let $\{E_k\}_{k \in \N}$ be
a countable collection of measurable sets
in ${\mathbb R}^n$.
Let $0 \le \theta<1$.
Suppose in addition that $p$ and $\kappa$ are positive
numbers such that $p\kappa>1$ and that $p \le 1$.
Then
\[
\left\|\sum_{k \in \N}(M\chi_{E_k})^\kappa\right\|_{L^p([-1,1]^n)}
\lesssim
\left(
\sum_{l=1}^\infty
\left(
2^{-\frac{n\theta}{p}}
\left\|\sum_{k \in \N}\chi_{E_k}
\right\|_{L^{p}([-2^l,2^l]^n)}
\right)^{\frac{1}{\kappa}}\right)^{\kappa}.
\]
\end{lemma}

\begin{proof}
By the scaling, we have
\[
\left\|\sum_{k \in \N}(M\chi_{E_k})^\kappa\right\|_{L^p([-1,1]^n)}
=\left(
\left\|
\left(\sum_{k \in \N}(M\chi_{E_k})^\kappa
\right)^{\frac{1}{\kappa}}
\right\|_{L^{p\kappa}([-1,1]^n)}\right)^{\kappa}.
\]
Recall that $(M\chi_{[-1,1]^n})^{\frac{\theta}{p\kappa}}$
is an $A_1$-weight of Muckenhoupt and Wheeden.
By the weighted vector-valued inequality obtained by Andersen and John \cite{AndersenJohn},
we obtain
\begin{align*}
\left\|\sum_{k \in \N}(M\chi_{E_k})^\kappa\right\|_{L^p([-1,1]^n)}
&\le\left(
\left\|
\left(\sum_{k \in \N}(M\chi_{E_k})^\kappa
\right)^{\frac{1}{\kappa}}(M\chi_{[-1,1]^n})^{\frac{\theta}{p\kappa}}
\right\|_{L^{p\kappa}({\mathbb R}^n)}\right)^{\kappa}\\
&\lesssim
\left(
\left\|
\left(\sum_{k \in \N}\chi_{E_k}
\right)^{\frac{1}{\kappa}}(M\chi_{[-1,1]^n})^{\frac{\theta}{p\kappa}}
\right\|_{L^{p\kappa}({\mathbb R}^n)}\right)^{\kappa}\\
&\lesssim
\left(
\sum_{l=1}^\infty
\frac{1}{2^{\frac{n\theta}{p\kappa}}}
\left\|
\left(\sum_{k \in \N}\chi_{E_k}
\right)^{\frac{1}{\kappa}}
\right\|_{L^{p\kappa}([-2^l,2^l]^n)}\right)^{\kappa}\\
&=
\left(
\sum_{l=1}^\infty
\left(
2^{-\frac{n\theta}{p}}
\left\|\sum_{k \in \N}\chi_{E_k}
\right\|_{L^{p}([-2^l,2^l]^n)}
\right)^{\frac{1}{\kappa}}\right)^{\kappa},
\end{align*}
as was to be shown.
\end{proof}

\subsection{Structure of generalized Hardy-Morrey spaces}

The grand maximal operator
characterizes
Hardy-Morrey spaces
defined by the norm (\ref{HM}).
To formulate the result,
we recall the following two fundamental notions.
\begin{enumerate}
\item
Topologize ${\mathcal S}({\mathbb R}^n)$ by norms $\{p_N\}_{N \in {\mathbb N}}$
given by
$$
 p_N(\varphi)
\equiv
\sum_{|\alpha|\le N}
\sup_{x\in{\mathbb R}^n}(1+|x|)^N
|\partial^{\alpha}\varphi(x)|
$$
for each $N \in {\mathbb N}$.
Define
${\mathcal F}_N\equiv
\{\varphi\in{\mathcal S}({\mathbb R}^n):p_N(\varphi)\le 1\}$.
\item
Let $f \in {\mathcal S}'({\mathbb R}^n)$.
The grand maximal operator ${\mathcal M}f$
is given by
\begin{equation}\label{eq:gm}
{\mathcal M}f(x)\equiv
\sup
\{|t^{-n}\psi(t^{-1}\cdot)*f(x)|
\,:\,t>0, \, \psi \in {\mathcal F}_N\}\quad (x \in {\mathbb R}^n),
\end{equation}
where we choose and fix a large integer $N$.
\end{enumerate}

In analogy to \cite[Section 3]{NaSa2012},
we can prove the following proposition.
\begin{proposition}
Let $0<p \le 1$ and $\phi \in {\mathcal G}_p$.
Then
\[
\|f\|_{H{\mathcal M}_{p,\phi}} \sim \|{\mathcal M}f\|_{{\mathcal M}_{p,\phi}}
\]
for all $f \in {\mathcal S}'({\mathbb R}^n)$.
\end{proposition}

From this proposition,
we can use the norm
$\|{\mathcal M}f\|_{{\mathcal M}_{p,\phi}}$
for the space ${\mathcal M}_{p,\phi}({\mathbb R}^n)$.

\begin{lemma}
Let $0<p \le 1$ and $\phi \in {\mathcal G}_p$.
Then
$H{\mathcal M}_{p,\phi}({\mathbb R}^n)$
is continuously embedded into
${\mathcal S}'({\mathbb R}^n)$.
\end{lemma}

\begin{proof}
Let $N \gg 1$ be fixed.
Then there exists a constant $C>0$ such that
if $|y| \le 1$ and $\varphi \in {\mathcal F}_N$,
then $C^{-1}\varphi(\cdot-y) \in {\mathcal F}_N$.
Thus,
\[
|\langle f,\varphi \rangle|
\lesssim \inf_{|y| \le 1}{\mathcal M}f(y)
\]
for all $\varphi \in {\mathcal F}_N$.
This implies
\[
|\langle f,\varphi \rangle|
\lesssim \|f\|_{H{\mathcal M}_{p,\phi}},
\]
as was to be shown.
\end{proof}

Going through the same argument as \cite[Theorem 3.4]{NaSa2012},
we obtain the following theorem,
whose proof will be omitted.
\begin{theorem}
Let $N \gg 1$.
Then
$f \in H{\mathcal M}_{p,\phi}({\mathbb R}^n)$
if and only if
${\mathcal M}f \in {\mathcal M}_{p,\phi}({\mathbb R}^n)$.
\end{theorem}

\begin{remark}
When $1<p<\infty$ and $\phi \in {\mathcal G}_p$,
$M$ is bounded on ${\mathcal M}_{p,\phi}({\mathbb R}^n)$
as Proposition \ref{nakaiVagif0} implies.
Also, similarly to \cite{GHS-pre},
we can show that ${\mathcal M}_{p,\phi}({\mathbb R}^n)$ is realized
as a dual space of a Banach space.
By combining these facts,
we can show the following fact for $f \in {\mathcal S}'({\mathbb R}^n)$.
The distribution $f$ is represented by a function in ${\mathcal M}_{p,\phi}({\mathbb R}^n)$
if and only if 
${\mathcal M}f \in {\mathcal M}_{p,\phi}({\mathbb R}^n)$.
\end{remark}

\section{Atomic decomposition}
\label{s4}

We return to the case where $\varphi$ is independent of $x$
and we prove the remaining theorems.
\subsection{Proof of Theorem \ref{thm:131108-2}}

Prior to the proof,
we remark that
(\ref{eq:140327-11156}) implies 
\begin{equation}\label{eq:140327-1115}
\int_r^\infty \frac{\phi(x,s)^p}{\eta(x,s)^ps}\,ds
\le C
\frac{\phi(x,r)^p}{\eta(x,r)^p}
\end{equation}
and 
\begin{equation}\label{eq:140327-3}
\int_r^\infty \phi(x,s)\frac{ds}{s} \le C
\phi(x,r).
\end{equation}
For the proof of (\ref{eq:140327-1115}),
we refer to \cite{Nakai94}. 

We start with collecting auxiliary estimates.
\begin{lemma}
Let $p,\eta,a_j,Q_j$ be the same as Theorem \ref{thm:131108-2}.
Then
\begin{equation}\label{eq:140327-1113}
\|Ma_j\|_{{\mathcal M}_{p,\eta}} \lesssim
\frac{1}{\eta(Q_j)}.
\end{equation}
\end{lemma}

\begin{proof}
When $p<1$.
we use the boundedness
of the Hardy-Littlewood maximal operator
$M:{\mathcal M}_{1,\eta}({\mathbb R}^n) \to {\mathcal M}_{p,\eta}({\mathbb R}^n)$; 
\cite{SST12-2} for more details.
When $p=1$,
this can be replaced by the
${\mathcal M}_{q,\eta}({\mathbb R}^n)$-boundedness of $M$.
Using this boundedness and (\ref{eq:140327-1111})
and
(\ref{eq:thm1-1-a}),
we obtain (\ref{eq:140327-1113}).
\end{proof}

Note that (\ref{eq:140327-1113}) readily yields
\begin{equation}\label{eq:140327-1113a}
\|(Ma_j)^p\|_{{\mathcal M}_{1,\eta^p}} 
=
(\|Ma_j\|_{{\mathcal M}_{p,\eta}})^p 
\lesssim
\frac{1}{\eta(Q_j)^p}.
\end{equation}

We invoke an estimate from \cite{NaSa2012};
\begin{equation}\label{eq:150308-71}
{\mathcal M}a_j(x)
\lesssim
\chi_{3Q_j}(x)Ma_j(x)
+
\frac{\ell(Q_j)^{n+d+1}}{
\ell(Q_j)^{n+d+1}+|x-c(Q_j)|^{n+d+1}}.
\end{equation}
	
For the time being,
we assume that there exists $N \in {\mathbb N}$
such that $\lambda_j=0$ whenever $j \ge N$.

Observe first that
\[
{\mathcal M}f(x)
\le
\sum_{j=1}^\infty |\lambda_j|{\mathcal M}a_j(x)
\le
\left(
\sum_{j=1}^\infty |\lambda_j|^p{\mathcal M}a_j(x)^p
\right)^{\frac{1}{p}},
\]
since ${\mathcal M}$ is sublinear and $0<p \le 1$.
Set
\[
\tau\equiv \frac{n+d+1}{n} \in (1,\infty).
\]
Consequently from (\ref{eq:150308-71}),
we obtain
\begin{align*}
{\mathcal M}f(x)
&\lesssim
\left(
\sum_{j=1}^\infty |\lambda_j|^p\chi_{3Q_j}(x)Ma_j(x)^p
\right)^{\frac{1}{p}}\\
&\quad +
\left(
\sum_{j=1}^\infty
\frac{|\lambda_j|^p\ell(Q_j)^{p(n+d+1)}}{
\ell(Q_j)^{p(n+d+1)}+|x-c(Q_j)|^{p(n+d+1)}}
\right)^{\frac{1}{p}}\\
&\lesssim
\left(
\sum_{j=1}^\infty |\lambda_j|^p\chi_{3Q_j}(x)Ma_j(x)^p
\right)^{\frac{1}{p}}
+
\left(
\sum_{j=1}^\infty |\lambda_j|^p
M\chi_{3Q_j}(x)^{p\tau}
\right)^{\frac{1}{p}}.
\end{align*}
Thus,
by the quasi-triangle inequality,
\begin{align*}
\lefteqn{
\|{\mathcal M}f\|_{{\mathcal M}_{p,\phi}}
}\\
&\lesssim
\left\|\left(
\sum_{j=1}^\infty (|\lambda_j|\chi_{3Q_j}Ma_j)^p
\right)^{\frac{1}{p}}\right\|_{{\mathcal M}_{p,\phi}}+
\left\|\left(
\sum_{j=1}^\infty |\lambda_j|^p
(M\chi_{3Q_j})^{p\tau}
\right)^{\frac{1}{p}}\right\|_{{\mathcal M}_{p,\phi}}\\
&=
\left(
\left\|
\sum_{j=1}^\infty (|\lambda_j|\chi_{3Q_j}Ma_j)^p
\right\|_{{\mathcal M}_{1,\phi^p}}
\right)^{\frac{1}{p}}+
\left(
\left\|\left(
\sum_{j=1}^\infty |\lambda_j|^p
(M\chi_{3Q_j})^{p\tau}
\right)^{\frac{1}{p\tau}}\right\|_{{\mathcal M}_{p\tau,\phi^\tau}}
\right)^{\tau}.
\end{align*}
Note that (\ref{eq:140327-3}) and (\ref{eq:140327-1115}) allows us
to use Theorem \ref{lem:131108-2} and Corollary \ref{cor:140327-111},
respectively.
Thus, we obtain
\begin{align*}
\|{\mathcal M}f\|_{{\mathcal M}_{p,\phi}}
&\lesssim
\left\|\left(
\sum_{j=1}^\infty (|\lambda_j|\chi_{Q_j})^p
\right)^{\frac{1}{p}}\right\|_{{\mathcal M}_{p,\phi}}.
\end{align*}
Thus, we obtain the desired result.

\subsection{Proof of Theorem \ref{thm:2}}

We define the topology on $\mathcal{S}(\mathbb{R}^n)$ with the norm
$\{\rho_N\}_{N\in \mathbb{N}}$ which is given by the following formula:
\[
\rho_{N}(\varphi)\equiv \sum\limits_{|\alpha|\leq N}
\sup\limits_{x \in \mathbb{R}^n} (1+|x|)^N|\partial^{\alpha}\varphi(x)|.
\]
We define
\begin{equation}\label{eq:FN}
\mathcal{F}_{N} \equiv \{\varphi \in \mathcal{S}(\mathbb{R}^n):\rho_{N}(\varphi)\leq 1\}.
\end{equation}

\begin{definition}\label{gm}
The grand maximal operator $\mathcal{M}f$ is defined by
\[
\mathcal{M}f(x)\equiv \sup \{|t^{-n} \varphi (t^{-1} \cdot ) *f(x)|: t>0, \varphi \in \mathcal{F}_N \}
\]
for all $f\in \mathcal{S}'(\mathbb{R}^n)$ and $x\in \mathbb{R}^n$.
\end{definition}

The following proposition can be proved
similar to \cite[Section 3]{NaSa2012}.

We invoke the following lemma.
By $C^{\infty}_{{\rm comp}}({\mathbb R}^n)$
we denote the set of all compactly supported smooth functions
in ${\mathbb R}^n$.
We refer to \cite{Stein1993} for the proof.
\begin{lemma}\label{lem:CZd}
Let $f \in {\mathcal S}'({\mathbb R}^n)$,
$d \in \{0,1,2,\cdots\}$ and $j \in {\mathbb Z}$.
Then there exist
collections of cubes $\{Q_{j,k}\}_{k \in K_j}$
and functions
$\{\eta_{j,k}\}_{k \in K_j}\subset C^{\infty}_{{\rm comp}}({\mathbb R}^n)$,
which are all indexed by a set $K_j$ for every $j$,
and a decomposition
\[
f=g_j+b_j, \quad b_j=\sum_{k\in K_j}b_{j,k},
\]
such that{\rm:}
\begin{enumerate}
\item[\rm(0)]
$g_j,b_j, b_{j,k} \in {\mathcal S}'({\mathbb R}^n)$.
\item[\rm(i)]
Define
${\mathcal O}_j\equiv \{y \in {\mathbb R}^n:{\mathcal M} f(y)>2^j\}$
and consider its Whitney decomposition.
Then the cubes $\{Q_{j,k}\}_{k \in K_j}$
have the bounded intersection property, and
\begin{equation}\label{eq:120922-1}
{\mathcal O}_j
=
\bigcup_{k \in K_j} Q_{j,k}.
\end{equation}
\item[\rm(ii)]
Consider the partition of unity
$\{\eta_{j,k}\}_{k \in K_j}$
with respect to
$\{Q_{j,k}\}_{k \in K_j}$.
Then each function $\eta_{j,k}$ is supported in $Q_{j,k}$ and
$$
 \sum_{k \in K_j}\eta_{j,k}
=\chi_{\{y\in{\mathbb R}^n\,:\,{\mathcal M} f(y)>2^j\}},
\quad 0\le\eta_{j,k}\le1.
$$
\item[\rm(iii)]
The distribution $g_j$ satisfies the inequality:
\begin{equation}\label{eq:120919-13}
 {\mathcal M}g_j(x)
 \lesssim
 {\mathcal M} f(x)\chi_{{\mathcal O}_j{}^c}(x)+
 2^j\sum_{k \in K_j}
\frac{{\ell_{j,k}}^{n+d+1}}{(\ell_{j,k}+|x-x_{j,k}|)^{n+d+1}}
\end{equation}
for all $x \in {\mathbb R}^n$.
\item[\rm(iv)]
Each distribution $b_{j,k}$ is given by $b_{j,k}=(f-c_{j,k})\eta_{j,k}$
with a certain polynomial $c_{j,k}\in{\mathcal P}_d({\mathbb R}^n)$ satisfying
\[
\langle f-c_{j,k},\eta_{j,k} \cdot q \rangle=0
\mbox{ for all }
q\in{\mathcal P}_d({\mathbb R}^n),
\]
and
\begin{equation}\label{eq:120919-14}
 {\mathcal M}b_{j,k}(x)
 \lesssim
 {\mathcal M} f(x) \chi_{Q_{j,k}}(x) +
 2^j\cdot\frac{{\ell_{j,k}}^{n+d+1}}{|x-x_{j,k}|^{n+d+1}}
\chi_{{\mathbb R}^n\setminus Q_{j,k}}(x)
\end{equation}
for all $x \in {\mathbb R}^n$.
\end{enumerate}
In the above,
$x_{j,k}$ and $\ell_{j,k}$ denote
the center and the side-length of $Q_{j,k}$, respectively,
and the implicit constants
are dependent only on $n$.
If we assume $f \in L^1_{\rm loc}({\mathbb R}^n)$ in addition,
then we have
\begin{equation}\label{eq:150302-21}
\|g_j\|_{L^\infty} \lesssim 2^{-j}.
\end{equation}

\end{lemma}

\begin{lemma}
Let $\varphi \in {\mathcal S}({\mathbb R}^n)$
and $\theta \in (0,1)$.
Keep to the same notation as Lemma \ref{lem:CZd}.
Then we have
\begin{equation}\label{eq:120919-11}
|\langle b_{j},\varphi \rangle|
\le C_{\varphi,\theta}
\left(
\sum_{l=1}^\infty
\left(
2^{-\frac{n\theta}{p}}
\left\|
 {\mathcal M} f \cdot \chi_{{\mathcal O}_j}
\right\|_{L^{p}([-2^l,2^l]^n)}
\right)^{\frac{1}{\kappa}}\right)^{\kappa}
\end{equation}
and
\begin{equation}\label{eq:120919-12}
|\langle g_{j},\varphi \rangle|
\le C_{\varphi,\theta}
\left(
\sum_{l=1}^\infty
\left(
2^{-\frac{n\theta}{p}}
\left\|
2^j \chi_{{\mathcal O}_j}
\right\|_{L^{p}([-2^l,2^l]^n)}
\right)^{\frac{1}{\kappa}}\right)^{\kappa}
+C_{\varphi,\theta}
\| {\mathcal M} f \cdot \chi_{{\mathcal O}_j^c}\|_{L^p([-1,1]^n)},
\end{equation}
where the constants $C_{\varphi,\theta}$ in $(\ref{eq:120919-11})$
and $(\ref{eq:120919-12})$ depend on $\varphi$ and $\theta$
but not on $j$ or $k$.
\end{lemma}

\begin{proof}
For some large constant $M \equiv M_\varphi$,
we have $\psi_x \equiv M^{-1}\varphi(x-\cdot) \in {\mathcal F}_N$
for all $x \in [-1,1]^n$,
so that
\[
|\langle b_{j},\varphi \rangle|
=
|b_j*\psi_x(z)|_{z=x}
\le M
\inf_{x \in [-1,1]^n}{\mathcal M}b_j(x).
\]
Thus, we have
\[
|\langle b_{j},\varphi \rangle|
\lesssim
\inf_{x \in [-1,1]^n}{\mathcal M}b_j(x)
\lesssim
\inf_{x \in [-1,1]^n}
\sum_{k \in K_j}{\mathcal M}b_{j,k}(x).
\]
Observe also that
\[
M\chi_Q(x) \gtrsim \frac{|Q|}{|Q|+|x-x_Q|^n}
\ge \frac{|Q|}{|x-x_Q|^n}\chi_{{\mathbb R}^n \setminus Q}(x)
\quad(x \in {\mathbb R}^n),
\]
if $Q$ is a cube centered at $x_Q$.
It follows from (\ref{eq:120919-14})
that
\begin{align}
\sum_{k \in K_j}{\mathcal M}b_{j,k}(x)
&\lesssim
\sum_{k \in K_j}\left(
 {\mathcal M} f(x) \chi_{Q_{j,k}}(x) +
 2^j\cdot\frac{{\ell_{j,k}}^{n+d+1}}{|x-x_{j,k}|^{n+d+1}}
\chi_{{\mathbb R}^n\setminus Q_{j,k}}(x)
\right)\nonumber\\
\label{eq:140327-2}
&\lesssim
 {\mathcal M} f(x) \chi_{{\mathcal O}_j}(x) +
 2^j
\sum_{k \in K_j}M\chi_{Q_{j,k}}(x)^{\frac{n+d+1}{n}}.
\end{align}
We abbreviate $\kappa\equiv \dfrac{n+d+1}{n}$.
From (\ref{eq:140327-2}),
we deduce
\begin{align*}
\|{\mathcal M}b_j\|_{L^p([-1,1]^n)}
&\lesssim
\left\|
 {\mathcal M} f \cdot \chi_{{\mathcal O}_j} +
 2^j
\sum_{k \in K_j}(M\chi_{Q_{j,k}})^{\kappa}
\right\|_{L^p([-1,1]^n)}\\
&\lesssim
\left\|
 {\mathcal M} f \cdot \chi_{{\mathcal O}_j}
\right\|_{L^p([-1,1]^n)} +
\left\| 2^j
\sum_{k \in K_j}(M\chi_{Q_{j,k}})^{\kappa}
\right\|_{L^p([-1,1]^n)}.
\end{align*}
By using (\ref{lem:140327-1}),
we obtain
\begin{align*}
\lefteqn{
\|{\mathcal M}b_j\|_{L^p([-1,1]^n)}
}\\
&
\lesssim
\left\|
 {\mathcal M} f \cdot \chi_{{\mathcal O}_j}
\right\|_{L^p([-1,1]^n)}
+
\left(
\sum_{l=1}^\infty
\left(
2^{-\frac{n\theta}{p}}
\left\|
2^j \chi_{{\mathcal O}_j}
\right\|_{L^{p}([-2^l,2^l]^n)}
\right)^{\frac{1}{\kappa}}\right)^{\kappa}.
\end{align*}
So, the estimate for the first term is valid.

In the same way we can prove
(\ref{eq:120919-12}).
Indeed, by using the Fefferman-Stein inequality
for $A_1$-weighted Lebesgue spaces \cite{AndersenJohn},
we obtain
\begin{eqnarray*}
&&
\|{\mathcal M}g_j\|_{L^p([-1,1]^n)}\\
&&\lesssim\left\|
 {\mathcal M} f \cdot \chi_{{\mathcal O}_j{}^c}
\right\|_{L^p([-1,1]^n)}
+
\left\| \sum_{k \in K_j}
\frac{2^j\cdot{\ell_{j,k}}^{n+d+1}}{(\ell_{j,k}+|x-x_{j,k}|)^{n+d+1}}
\right\|_{L^p([-1,1]^n)}\\
&&\lesssim\left\|
 {\mathcal M} f \cdot \chi_{{\mathcal O}_j{}^c}
\right\|_{L^p([-1,1]^n)}+
\left\| \sum_{k \in K_j}2^j(M\chi_{Q_{j,k}})^{\frac{n+d+1}{n}}
\right\|_{L^p([-1,1]^n)}\\
&&\lesssim\left\|
 {\mathcal M} f \cdot \chi_{{\mathcal O}_j{}^c}
\right\|_{L^p([-1,1]^n)}
+
\left(
\sum_{l=1}^\infty
\left(
2^{-\frac{n\theta}{p}}
\left\|
2^j \chi_{{\mathcal O}_j}
\right\|_{L^{p}([-2^l,2^l]^n)}
\right)^{\frac{1}{\kappa}}\right)^{\kappa}.
\end{eqnarray*}
Thus, (\ref{eq:120919-12}) is proved.
\end{proof}

The key observation is the following.
\begin{lemma}\label{lem:131127-1}
Assume $(\ref{eq:140327-3})$.
In the notation of Lemma {\rm\ref{lem:CZd}},
in the topology of ${\mathcal S}'({\mathbb R}^n)$,
we have
$g_j \to 0$ as $j \to -\infty$
and
$b_j \to 0$ as $j \to \infty$.
In particular,
\[
f=\sum_{j=-\infty}^\infty (g_{j+1}-g_j)
\]
in the topology of ${\mathcal S}'({\mathbb R}^n)$.
\end{lemma}

\begin{proof}
Let us show that $b_j \to 0$ as $j \to \infty$
in ${\mathcal S}'({\mathbb R}^n)$.
Once this is proved,
then we have
$f=\lim_{j \to \infty}g_j$
in ${\mathcal S}'({\mathbb R}^n)$.
Let us choose a test function $\varphi \in {\mathcal S}({\mathbb R}^n)$.
Then we have
\[
|\langle b_j,\varphi \rangle|
\lesssim\inf_{x \in [-1,1]^n}
{\mathcal M}b_j(x)
\lesssim
\|{\mathcal M}b_j\|_{L^p([-1,1]^n)},
\]
where the implicit constant does depend on $\varphi$.

Assume (\ref{eq:140327-3})
and choose $\theta>0$ so that $\tau<\theta<1$.
Note that
\[
\|f\|_{H{\mathcal M}_{p,\phi}}
\ge
\left\|
 {\mathcal M} f \cdot \chi_{{\mathcal O}_j}
\right\|_{L^{p}([-2^l,2^l]^n)} \to 0
\quad (j \to \infty)
\]
and that
\begin{align*}
\left(
\sum_{l=1}^\infty
\left(
2^{-\frac{n\theta}{p}}
\left\|
 {\mathcal M} f \cdot \chi_{{\mathcal O}_j}
\right\|_{L^{p}([-2^l,2^l]^n)}
\right)^{\frac{1}{\kappa}}\right)^{\kappa}
&\le
\left(
\sum_{l=1}^\infty
\left(
\varphi(2^l)2^{\frac{n(1-\theta)}{p}}\|f\|_{H{\mathcal M}_{p,\phi}}
\right)^{\frac{1}{\kappa}}\right)^{\kappa}\\
&=C_0\|f\|_{H{\mathcal M}_{p,\phi}}.
\end{align*}

Hence it follows from (\ref{eq:120919-11}) that
$\langle b_j,\varphi \rangle \to 0$
as $j \to \infty$.
Likewise by using (\ref{eq:120919-12}),
we obtain
\[
|\langle g_j,\varphi \rangle|
\lesssim
\left(
\sum_{l=1}^\infty
\left(
2^{-\frac{n\theta}{p}}
\left\|2^j \cdot \chi_{{\mathcal O}_j}+
 {\mathcal M} f \cdot \chi_{({\mathcal O}_j)^c}
\right\|_{L^{p}([-2^l,2^l]^n)}
\right)^{\frac{1}{\kappa}}\right)^{\kappa}.
\]
Hence, $g_j \to 0$ as $j \to -\infty$ by the Lebesgue convergence theorem.
Consequently,
it follows that
$$f=\lim_{j \to \infty}g_j=\lim_{j,k \to \infty}\sum_{l=-k}^j(g_{l+1}-g_l)$$
in ${\mathcal S}'({\mathbb R}^n)$.
\end{proof}

We prove Theorem \ref{thm:2} when $f \in L^1_{\rm loc}({\mathbb R}^n)$.
\begin{proof}[Proof of Theorem \ref{thm:2} when $f \in L^1_{\rm loc}({\mathbb R}^n)$]
For each $j\in{\mathbb Z}$,
consider the level set
\begin{equation}\label{eq:O-j}
{\mathcal O}_j\equiv \{x \in {\mathbb R}^n :{\mathcal M}f(x)>2^j\}.
\end{equation}
Then it follows immediately from the definition that
\begin{equation}\label{eq:O-j+1}
{\mathcal O}_{j+1}\subset{\mathcal O}_j.
\end{equation}
If we invoke Lemma~\ref{lem:CZd},
then $f$ can be decomposed;
$$
 f=g_j+b_j, \quad b_j=\sum_kb_{j,k}, \quad b_{j,k}=(f-c_{j,k})\eta_{j,k},
$$
where each $b_{j,k}$ is supported in a cube $Q_{j,k}$
as is described in Lemma \ref{lem:CZd}.

We know that
\begin{equation}
 f=\sum_{j=-\infty}^\infty (g_{j+1}-g_j),
\end{equation}
with the sum converging in the sense of distributions
from Lemma \ref{lem:131127-1}.
Here,
going through the same argument as the one in \cite[pp. 108--109]{Stein1993},
we have an expression;
\begin{equation}\label{fgeA}
 f=\sum_{j,k}A_{j,k},
\quad g_{j+1}-g_j=\sum_k A_{j,k} \quad (j \in {\mathbb Z})
\end{equation}
in the sense of distributions,
where each $A_{j,k}$,
supported in $Q_{j,k}$,
satisfies the pointwise estimate
$|A_{j,k}(x)|\le C_02^j$
for some universal constant $C_0$
and the moment condition
$\displaystyle \int_{{\mathbb R}^n} A_{j,k}(x)q(x)\,dx=0$
for every $q\in{\mathcal P}_d({\mathbb R}^n)$.
With these observations in mind,
let us set
$$
 a_{j,k}\equiv \frac{A_{j,k}}{C_02^j},
 \quad
 \kappa_{j,k}\equiv C_02^j.
$$
Then we automatically obtain
that each $a_{j,k}$ satisfies
\[
|a_{j,k}| \le \chi_{Q_{j,k}}, \quad
\int_{{\mathbb R}^n}x^\alpha a_{j,k}(x)\,dx=0
\quad (|\alpha| \le L)
\]
and that $\displaystyle f=\sum_{j,k}\kappa_{j,k}a_{j,k}$
in the topology of $H{\mathcal M}_{p,\phi}({\mathbb R}^n)$,
once we prove the estimate of coefficients.
Rearrange $\{a_{j,k}\}$ and so on
to obtain $\{a_j\}$ and so on.

To establish (\ref{eq:thm2-1})
we need to estimate
\begin{align*}
\alpha
\equiv
\left\|\left(
\sum_{j=-\infty}^\infty |\lambda_j \chi_{Q_j}|^v
\right)^{1/v}\right\|_{{\mathcal M}_{p,\phi}}.
\end{align*}
Since
$\{(\kappa_{j,k};Q_{j,k})\}_{j,k}=
\{(\lambda_j;Q_j)\}_{j}$
as a set,
we have
\begin{align*}
\alpha=
\left\|\left(
\sum_{j=-\infty}^\infty
\sum_{k \in K_j} |\kappa_{j,k} \chi_{Q_{j,k}}|^v
\right)^{1/v}\right\|_{{\mathcal M}_{p,\phi}}.
\end{align*}
If we insert the definition
of $\kappa_j$,
then we have
\[
\alpha
=C_0
\left\|\left(
\sum_{j=-\infty}^\infty
\sum_{k \in K_j} |2^j \chi_{Q_{j,k}}|^v
\right)^{1/v}\right\|_{{\mathcal M}_{p,\phi}}
=C_0
\left\|\left(
\sum_{j=-\infty}^\infty
2^{jv}\sum_{k \in K_j} \chi_{Q_{j,k}}
\right)^{1/v}\right\|_{{\mathcal M}_{p,\phi}}.
\]
Observe that (\ref{eq:120922-1})
together with the bounded overlapping property
yields
\[
\chi_{{\mathcal O}_j}(x)
\le
\sum_{k \in K_j}
\chi_{Q_{j,k}}(x)
\lesssim
\chi_{{\mathcal O}_j}(x)
\quad (x \in {\mathbb R}^n).
\]
Thus, we have
\begin{align*}
\alpha
&\lesssim
\left\|\left(\sum_{j=-\infty}^\infty
 \left(2^j\chi_{{\mathcal O}_j}\right)^{v}
 \right)^{1/v}\right\|_{{\mathcal M}_{p,\phi}}.
\end{align*}
Recall that ${\mathcal O}_j \supset {\mathcal O}_{j+1}$
for each $j \in {\mathbb Z}$.
Consequently we have
\[
\sum_{j=-\infty}^\infty
 \left(2^j\chi_{{\mathcal O}_j}(x)\right)^{v}
\sim
 \left(
\sum_{j=-\infty}^\infty 2^j\chi_{{\mathcal O}_j}(x)\right)^{v}
\sim
 \left(
\sum_{j=-\infty}^\infty
2^j\chi_{{\mathcal O}_j \setminus {\mathcal O}_{j+1}}(x)
\right)^{v}
\quad (x \in {\mathbb R}^n).
\]
Thus,
we obtain
\begin{align*}
\alpha
\lesssim
 \left\|
\sum_{j=-\infty}^\infty
2^j\chi_{{\mathcal O}_j \setminus {\mathcal O}_{j+1}}
\right\|_{{\mathcal M}_{p,\phi}}.
\end{align*}

It follows from the definition of ${\mathcal O}_j$
that
we have $2^j<{\mathcal M}f(x)$ for all $x \in {\mathcal O}_j$.
Hence, we have
\begin{align*}
\alpha
\lesssim
 \left\|
\sum_{j=-\infty}^\infty
\chi_{{\mathcal O}_j \setminus {\mathcal O}_{j+1}}
{\mathcal M}f
\right\|_{{\mathcal M}_{p,\phi}}
=\|{\mathcal M}f\|_{{\mathcal M}_{p,\phi}}=\|f\|_{H{\mathcal M}_{p,\phi}}.
\end{align*}
This is the desired result.
\end{proof}

\begin{proof}[Proof of Theorem \ref{thm:2} for general cases]
According to (\ref{eq:120919-13}),
$g_j$ is a locally integrable function and
it satisfies
$\|g_j\|_{H{\mathcal M}_{p,\phi}}
\lesssim
\|f\|_{H{\mathcal M}_{p,\phi}}$.
Therefore, applying the above paragraph,
we see that each $g_j$ has a decomposition;
there exist
a collection $\{Q_{l,j}\}_{l=1}^\infty$ of cubes,
$\{a_{l,j}\}_{l=1}^\infty \subset L^\infty({\mathbb R}^n)$,
and
$\{\lambda_{l,j}\}_{l=1}^\infty \subset [0,\infty)$
such that
\begin{equation}\label{eq:150302-1}
g_j=\sum_{l=1}^\infty \lambda_{l,j}a_{l,j}
\end{equation}
unconditionally in ${\mathcal S}'({\mathbb R}^n)$,
that
$|a_{l,j}| \le \chi_{Q_{l,j}}$,
that
\[
\int_{{\mathbb R}^n}a_{Q,j}(x)x^\alpha\,dx=0
\]
for all $|\alpha| \le L$
and that
\begin{equation}\label{eq:150302-2}
\left\|\left(\sum_{l=1}^\infty
(\lambda_{j,l}\chi_{Q_{j,l}})^{v}
\right)^{1/v}\right\|_{{\mathcal M}_{p,\phi}}
\le C_v\|f\|_{H{\mathcal M}_{p,\phi}}.
\end{equation}
We may assume that each $Q_{j,l}$
is realized as $3Q$ for some dyadic cube $Q$.
Since $v \le 1$, by using $a^v+b^v \ge (a+b)^v$ for $a,b \ge 0$
and taking into account the case when $Q_{j,l}=Q_{j',l'}$
for some $(j.l) \ne (j',l')$,
we have a decomposition
there exist
a collection 
$\{Q_{Q,j}\}_{Q \in {\mathcal D}}$ of cubes,
$\{a_{Q,j}\}_{Q \in {\mathcal D}} \subset L^\infty({\mathbb R}^n)$,
and
$\{\lambda_{Q,j}\}_{Q \in {\mathcal D}} \subset [0,\infty)$
such that
\begin{equation}\label{eq:150308-11}
g_j=\sum_{Q \in {\mathcal D}} \lambda_{Q,j}a_{Q,j}
\end{equation}
in ${\mathcal S}'({\mathbb R}^n)$, that
\begin{equation}\label{eq:150302-3}
|a_{Q,j}| \le \chi_{3Q},
\end{equation}
that
\[
\int_{{\mathbb R}^n}a_{Q,j}(x)x^\alpha\,dx=0
\]
for all $|\alpha| \le L$
and that
\begin{equation}\label{eq:150302-4}
\left\|\left(\sum_{Q \in {\mathcal D}}
(\lambda_{Q,l}\chi_{Q})^{v}
\right)^{1/v}\right\|_{{\mathcal M}_{p,\phi}}
\le C_v\|f\|_{H{\mathcal M}_{p,\phi}}.
\end{equation}

Fix $Q \in {\mathcal D}$.
Since
$\{a_{Q,j}\}_{j=1}^\infty$
is a bounded sequence in $L^\infty({\mathbb R}^n)$
from (\ref{eq:150302-3}),
and
$\{\lambda_{Q,j}\}_{j=1}^\infty \subset [0,\infty)$
is a bounded sequence in ${\mathbb R}$
from (\ref{eq:150302-4}),
we can choose subsequences
$\{a_{Q,j_k}\}_{k=1}^\infty$
and
$\{\lambda_{Q,j_k}\}_{k=1}^\infty \subset [0,\infty)$
so that
$\{a_{Q,j_k}\}_{k=1}^\infty$
and
$\{\lambda_{Q,j_k}\}_{k=1}^\infty \subset [0,\infty)$
are convergent to
$a_Q$ and $\lambda_Q$ respectively,
where the convergence of
$\{a_{Q,j_k}\}_{k=1}^\infty$
takes place in the weak-* topology of $L^\infty({\mathbb R}^n)$.

Let us set
\begin{equation}\label{eq:150302-141}
g\equiv \sum_{Q \in {\mathcal D}}\lambda_Q a_Q.
\end{equation}
Then according to Theorem \ref{thm:131108-2},
we have
$g \in H{\mathcal M}_{p,\phi}({\mathbb R}^n)$.
By the Fatou lemma,
we can conclude the proof once we show that
\begin{equation}\label{eq:150302-14}
f=g.
\end{equation}
To this end, we take a test function $\varphi \in {\mathcal S}({\mathbb R}^n)$.
If we insert (\ref{eq:150302-1}) to $g_J$ and use (\ref{eq:120919-11}), we obtain
\[
\langle f,\varphi \rangle
=
\lim_{J \to \infty}
\langle g_J,\varphi \rangle
=
\lim_{J \to \infty}
\sum_{Q \in {\mathcal D}}
\lambda_{Q,J}\langle a_{Q,J},\varphi \rangle.
\]
If we can change the order of
$\displaystyle
\lim_{J \to \infty}
$
and
$\displaystyle
\sum_{Q \in {\mathcal D}}
$
in the most right-hand side of the above formula,
we have
\[
\langle f,\varphi \rangle
=
\lim_{J \to \infty}
\sum_{Q \in {\mathcal D}}
\lambda_{Q,J}\langle a_{Q,J},\varphi \rangle
=
\sum_{Q \in {\mathcal D}}
\lim_{J \to \infty}
\lambda_{Q,J}\langle a_{Q,J},\varphi \rangle
=
\sum_{Q \in {\mathcal D}}
\lambda_{Q}\langle a_{Q},\varphi \rangle
=
\langle g,\varphi \rangle,
\]
showing $f=g$.
Thus, we are left with the task of justifying
the change of the order of
$\displaystyle
\lim_{J \to \infty}
$
and
$\displaystyle
\sum_{Q \in {\mathcal D}}.
$
Let $\varphi^\dagger \in C^\infty_{\rm c}({\mathbb R}^n)$
satisfy $\chi_{B(1)} \le \varphi^\dagger \le \chi_{B(2)}$.
Since $\varphi \in {\mathcal S}({\mathbb R}^n)$,
by decomposing
$\varphi=\varphi\varphi^\dagger(R^{-1}\cdot)+
\varphi(1-\varphi^\dagger(R^{-1}\cdot))$,
and using the fact
that
$H{\mathcal M}_{p,\phi}({\mathbb R}^n)$
(defined via the grand maximal operator)
is continuously embedded in
${\mathcal S}'({\mathbb R}^n)$
as well as Theorem \ref{thm:131108-2},
we see that the contribution
of the function $\varphi(1-\varphi^\dagger(R^{-1}\cdot))$
can be made as small as we wish.
In fact,
\[
\sum_{Q \in {\mathcal D}}
|\lambda_{Q,J}\langle a_{Q,J},\varphi(1-\varphi^\dagger(R^{-1}\cdot)) \rangle|
=
O(R^{-1}),
\]
where the implicit constant do not depend on $J$.
This implies that we can and do assume that $\varphi$ is supported
in a compact set $K$.
Suppose that $K$ is contained in $Q(2^N)$ for some $N>0$.
Let us set
\begin{align*}
{\rm I}
&\equiv
\sup_{J}
\sum_{Q \in {\mathcal D}, Q \cap K\ne \emptyset, \ell(Q) \le 2^{-A}}
|\lambda_{Q,J}\langle a_{Q,J},\varphi \rangle|\\
{\rm II}
&\equiv
\sup_{J}
\sum_{Q \in {\mathcal D}, Q \cap K\ne \emptyset, 2^{-A}<\ell(Q) \le 2^{A}}
|\lambda_{Q,J}\langle a_{Q,J}-a_Q,\varphi \rangle|\\
{\rm III}
&\equiv
\sup_{J}
\sum_{Q \in {\mathcal D}, Q \cap K\ne \emptyset, 2^{A}<\ell(Q)}
|\lambda_{Q,J}\langle a_{Q,J},\varphi \rangle|,
\end{align*}
where $A>N$.
Then we have
\begin{align}\label{eq:150302-13}
\sup_{J}
\sum_{Q \in {\mathcal D}}
|\lambda_{Q,J}\langle a_{Q,J},\varphi \rangle|
\le
2{\rm I}+{\rm II}+2{\rm III}.
\end{align}
For $l \in {\mathbb Z}$,
denote by ${\mathcal D}_l$
the set of all dyadic cubes
$Q$ such that $|Q|=2^{-ln}$.
As for ${\rm I}$,
we use
$
|\langle a_{Q,J},\varphi \rangle|
\lesssim
\ell(Q)^{n+L+1}$
and if $l \ge N$,
\begin{equation}\label{eq:150302-11}
\sum_{Q \in {\mathcal D}_l}
\frac{|\lambda_{Q,J}|}{\phi({\rm o},Q(2^N))}
\le
\frac{2^{ln/p}}{\phi({\rm o},Q(2^N))}
\left\|
\sum_{Q \in {\mathcal D}_l}
\lambda_{Q,J}\chi_Q
\right\|_{L^p(Q(2^N))}
\lesssim
2^{ln/p}\|f\|_{H{\mathcal M}_{p,\phi}}
\end{equation}
and hence
$$
{\rm I}\lesssim
\sum_{Q \in {\mathcal D}, Q \cap K\ne \emptyset, \ell(Q) \le 2^{-A}}
\phi({\rm o},\ell(Q)))\ell(Q)^{n+L+1}
=
O(2^{-A(n+L+1-n/p)}).
$$
As for ${\rm III}$,
we use
\begin{equation}\label{eq:150302-12}
\phi({\rm o},2^A) \to 0 \quad (A \to \infty)
\end{equation}
and $0<p \le 1$ to have
\[
{\rm III}
\lesssim
\sum_{Q \in {\mathcal D}, Q \cap K\ne \emptyset, \ell(Q)>2^{A}}
|\lambda_{Q,J}|
\le
\phi(2^A)
\left\|\sum_{Q \in{\mathcal D}}
|\lambda_{Q,J}|\chi_Q\right\|_{{\mathcal M}_{p,\phi}}
\lesssim
\phi(2^A)
\|f\|_{H{\mathcal M}_{p,\phi}}.
\]
In view of (\ref{eq:150302-11}) and (\ref{eq:150302-12}),
we see that ${\rm I}$ and ${\rm III}$ contribute
little to the sum (\ref{eq:150302-13}).
With this in mind
we use the weak-* convergence to ${\rm II}$
to see (\ref{eq:150302-14}).
\end{proof}

Finally, we state a corollary to conclude this section.
\begin{corollary}
If $\phi \in {\mathcal G}_p$ satisfies $(\ref{eq:140327-3})$,
then
$H{\mathcal M}_{1,\phi}({\mathbb R}^n)$
is embedded into
${\mathcal M}_{1,\phi}({\mathbb R}^n)$.
\end{corollary}

\begin{proof}
Under the notation of Theorem \ref{thm:2},
we have
\[
\sum_{j=1}^\infty \lambda_j|a_j|
\le
\sum_{j=1}^\infty \lambda_j\chi_{Q_j}.
\]
(\ref{eq:thm2-1}) with $v=1$ guarantees that the right-hand side
belongs to
${\mathcal M}_{1,\phi}({\mathbb R}^n)$.
Hence
$\displaystyle
f(x)=\sum_{j=1}^\infty \lambda_j a_j(x)
$
converges for almost all $x \in {\mathbb R}^n$.
Observe also
\begin{align*}
\lefteqn{
\int_{{\mathbb R}^n}|\kappa(x)|\left(
\sum_{j=1}^\infty \lambda_j |a_j(x)|\right)\,dx
}\\
&\lesssim
\sup_{x \in {\mathbb R}^n}(1+|x|)^{2n/p+1}|\kappa(x)|
\int_{{\mathbb R}^n}
(1+|x|)^{-2n-1}\left(
\sum_{j=1}^\infty \lambda_j |a_j(x)|\right)\,dx\\
&\lesssim
\sup_{x \in {\mathbb R}^n}(1+|x|)^{2n/p+1}|\kappa(x)|
\sum_{j=1}^\infty
(1+j)^{-2n-1}
\int_{|x| \le j}
\sum_{j=1}^\infty \lambda_j |a_j(x)|\,dx\\
&\lesssim
\sup_{x \in {\mathbb R}^n}(1+|x|)^{2n/p+1}|\kappa(x)|
\sum_{j=1}^\infty
\phi(j)(1+j)^{-2n/p-1}
\left\|\sum_{j=1}^\infty
\lambda_j\chi_{Q_j}\right\|_{{\mathcal M}_{1,\phi}}\\
&\lesssim
\sup_{x \in {\mathbb R}^n}(1+|x|)^{2n/p+1}|\kappa(x)|
\left\|\sum_{j=1}^\infty
\lambda_j\chi_{Q_j}\right\|_{{\mathcal M}_{1,\phi}}.
\end{align*}
Thus, $f$ is represented by an $L^1_{\rm loc}({\mathbb R}^n)$-functions
and satisfies
$\displaystyle
f(x)=\sum_{j=1}^\infty \lambda_j a_j(x)
$
for almost all $x \in {\mathbb R}^n$.
\end{proof}

\subsection{Applications to the boundedness of the singular integral operators}

Going through the same argument
as \cite[Theorem 5.5]{NaSa2012} and \cite[Theorem 5.5]{NaSa-pre2013},
we can prove the following theorem;
\begin{theorem}\label{t5.4}
Let $\phi$ satisfy
$(\ref{eq:140327-3})$.
Let $k \in \mathcal{S}({\mathbb R}^n)$.
Write
\[
A_m
\equiv
\sup_{x \in {\mathbb R}^n}|x|^{n+m}|\nabla^m k(x)|
\quad (m \in {\mathbb N} \cup \{0\}).
\]
Define a convolution operator $T$ by
\[
T f(x) \equiv k*f(x) \quad (f \in {\mathcal S}'({\mathbb R}^n)).
\]
Then, $T$, restricted to $H{\mathcal M}_{p,\phi}({\mathbb R}^n)$,
is an
$H{\mathcal M}_{p,\phi}({\mathbb R}^n)$-bounded operator
and the norm depends only on $\|\mathcal{F}k\|_{L^\infty}$
and a finite number of collections $A_1,A_2,\ldots,A_N$
with $N$ depending only on $\phi$.
\end{theorem}

Once Theorem \ref{t5.4} is proved,
we can obtain the Littlewood-Paley decomposition
in the same way
as \cite[Theorem 5.7]{NaSa2012} and \cite[Theorem 5.10]{NaSa-pre2013}.
\begin{theorem}\label{thm:LP}
Let $0<p \le 1$.
Let $\phi \in {\mathcal G}_p$ satisfy
$(\ref{eq:140327-3})$.
Let $\varphi \in {\mathcal S}({\mathbb R}^n)$ be a function
which is supported on $B(0,4) \setminus B(0,1/4)$
and satisfies
\[
\sum_{j=-\infty}^\infty|\varphi_j(\xi)|^2>0
\]
for $\xi \in {\mathbb R}^n \setminus \{0\}$.
Then the following norm is an equivalent norm
of $H{\mathcal M}_{p,\phi}({\mathbb R}^n)${\rm:}
\begin{equation}
\|f\|_{\dot{\cal E}^0_{p,\phi,2}}
\equiv
\left\|
\left(
\sum_{j=-\infty}^\infty |\varphi_j(D)f|^2
\right)^{1/2}\right\|_{{\mathcal M}_{p,\phi}},
\quad f \in {\mathcal S}'({\mathbb R}^n).
\end{equation}
\end{theorem}
Once we obtain Theorem \ref{thm:LP},
we can prove the wavelet decomposition
and the atomic decomposition
as in \cite{s08-2,SaTa2007}.

Finally, we point out a mistake in our earlier paper \cite{IST14}.
\begin{remark}
The function $A_{j,k}$ in \cite[p. 162]{IST14} is not in $L^\infty({\mathbb R}^n)$
unless $f \in L^1_{\rm loc}({\mathbb R}^n)$.
Thus, the proof of \cite[Theorem 1.3]{IST14} is valid
only of $f \in H{\mathcal M}^p_q({\mathbb R}^n) \cap L^1({\mathbb R}^n)$,
where $H{\mathcal M}^p_q({\mathbb R}^n)=H{\mathcal M}_{q,\phi}({\mathbb R}^n)$
with $\phi(t)=t^{-n/p}$.
The gap will be closed by the technique described above.
\end{remark}

\section{Applications to the Olsen inequality}
\label{s5}

This is a bilinear estimate of $I_\alpha$,
which is nowadays called the Olsen inequality
\cite{Olsen95}.
Recall that
we define the fractional integral operator $I_{\alpha}$
with $0<\alpha<n$
by;
\begin{align*}
I_{\alpha}f(x)=\int_{\mathbb{R}^n} \frac{f(y)}{|x-y|^{n-\alpha}} \ dy
\end{align*}
for all suitable functions $f$ on $\R^n$.
Olsen's inequality is the inequality
of the form
\[
\|g \cdot I_\alpha f\|_Z
\lesssim
\|f\|_X
\|g\|_Y,
\]
where $X,Y,Z$ are suitable quasi-Banash spaces.
There is a vast amount of literatures
on Olsen inequalities;
see
\cite{EGU,SST12-2,SaSuTa11-1,SaSuTa11-2,SSG10,Sugano11,SuTa03,Tanaka10,UNW12-1}
for theoretical aspects
and
\cite{GRST14,GST12,GST13}
for applications to PDEs.

Here we will prove the following theorem.
\begin{theorem}
Let $0<p \le 1$ and $0<\alpha<n$ and define $q$ by:
\[
\frac{1}{p}-\frac{1}{q}=\frac{\alpha}{n-\lambda}.
\]
Then
\[
\|I_\alpha f\|_{H{\mathcal M}_{q,\lambda}}
\lesssim
\|f\|_{H{\mathcal M}_{p,\lambda}}
\]
for all $f \in H{\mathcal M}_{p,\lambda}$. In particular, if $q > 1$,
then
\[
\|I_\alpha f\|_{{\mathcal M}_{q,\lambda}}
\lesssim
\|f\|_{H{\mathcal M}_{p,\lambda}}
\]
for all $f \in H{\mathcal M}_{p,\lambda}$.
\end{theorem}

\begin{proof}
Argue as we did in \cite{SaSuTa09-2}
by using Theorem \ref{thm:LP}.
\end{proof}

\begin{theorem}
Let $0<p \le 1$ and $0<\alpha<n$ and define $q$ by:
\[
\frac{1}{p}-\frac{1}{q}=\frac{\alpha}{n-\lambda}.
\]
Assume that $q \ge 1$.
Let $g \in {\mathcal M}_{1,n-\alpha}({\mathbb R}^n)$.
Then then there exists a constant $C>0$ such that
\[
\|g \cdot I_\alpha f\|_{H{\mathcal M}_{p,\lambda}}
\lesssim
\|g\|_{{\mathcal M}_{1,n-\alpha}}
\cdot
\|f\|_{H{\mathcal M}_{p,\lambda}}
\]
for all $f \in H{\mathcal M}_{p,\lambda}$.
\end{theorem}

\begin{proof}
We argue as we did in \cite[Theorem 1.7]{IST14}.
\end{proof}

\section{Acknowledgement}

The research of V. Guliyev was partially supported 
by the grant of Science Development Foundation under the President of the Republic of Azerbaijan,
Grant EIF-2013-9(15)-46/10/1 and by the grant of Presidium Azerbaijan National Academy of Science 2015.
A. Akbulut was partially supported by the grant of Ahi Evran University Scientific Research Projects
(PYO.FEN.4003.13.004) and (PYO.FEN.4003/2.13.006).
This paper is written during the stay of Y. Sawano in Ahi Evran University and Beijing Normal University.
Y. Sawano is thankful to Ahi Evran University and Beijing Normal University for this support of the stay there.
The authors are indepted to Mr. S. Nakamura
for his hint to improve (\ref{eq:140327-1115}).


\begin{thebibliography}{99}
\bibitem{AkbGulMus1}
A. Akbulut, V.S. Guliyev and R. Mustafayev,
"On the Boundedness of the maximal operator and singular
integral operators in generalized Morrey spaces,"
\textit{Math. Bohem.} {\bf 137} (1) (2012), 27--43.

\bibitem{AndersenJohn}
K. F. Andersen and R. T. John,
"Weighted inequalities for vecter-valued maximal functions and singular
integrals,"
\textit{Studia Math}. {\bf 69} (1980), 19--31.

\bibitem{BurGogGulMus1}
V. Burenkov, A. Gogatishvili, V.S. Guliyev and R. Mustafayev,
"Boundedness of the fractional maximal operator in local Morrey-type spaces,"
\textit{Complex Var. Elliptic Equ.} {\bf 55} (8-10) (2010), 739--758.

\bibitem{CF}
F.~Chiarenza and M.~Frasca,
"Morrey spaces and Hardy-Littlewood maximal function,"
\textit{Rend. Mat.} {\bf 7} (1987), 273--279.

\bibitem{EGU}
Eridani, H. Gunawan and M.I. Utoyo,
``A characterization for fractional integral operators
on generalized Morrey spaces'',
\textit{Anal. Theory Appl.} {\bf 28} (2012),
No. 3, 263--267.

\bibitem{EGNS}
Eridani, H. Gunawan, E. Nakai and Y. Sawano,
Characterizations for the generalized fractional
integral operators on Morrey spaces,
\textit{Math. Ineq. and Appl.} {\bf 17}, No. 2 (2014), 761--777.


\bibitem{FeffSt}
C. Fefferman and E. Stein,
"Some maximal inequalities",
\textit{Amer J. Math.} {\bf 93} (1971), 105--115.

\bibitem{GHS-pre}
H. Gunawan, D. Hakim and Y. Sawano,
"Nonsmooth atomic decompositions for generalized Orlicz-Morrey spaces,"
to appear in {\textit Math. Nachr.}.

\bibitem{GRST14}
S.~Gala, A.M. Ragusa, Y.~Sawano, and H.~Tanaka,
Uniqueness criterion of weak solutions for the dissipative quasi-geostrophic equations in Orlicz-Morrey spaces,
\textit{Appl. Anal.} {\bf 93} (2014), no. 2, 356--368.

\bibitem{GST12}
S.~Gala, Y.~Sawano, and H.~Tanaka,
A new Beale-Kato-Majda criteria
for the 3D magneto-micropolar fluid equations in the Orlicz-Morrey space,
\textit{Math. Methods Appl. Sci.} {\bf 35} (2012), no. 11, 1321--1334.

\bibitem{GST13}
S.~Gala, Y.~Sawano and H.~Tanaka,
On the uniqueness of weak solutions of the 3D MHD equations
in the Orlicz-Morrey space,
\textit{Appl. Anal.} {\bf 92} (2013), no. 4, 776--783.

\bibitem{GulDoc}
 V.S. Guliyev,
"Integral operators on function spaces
on the homogeneous groups and
 on domains in $\Rn$". Doctor's degree dissertation,
\textit{Mat. Inst. Steklov, Moscow,} 1994, 329 pp. (in Russian)

\bibitem{GulBook}
V.S. Guliyev, "Function spaces, Integral Operators and Two Weighted
 Inequalities on Homogeneous Groups. Some Applications,"
 Cashioglu, Baku, 1999, 332 pp. (in Russian)


\bibitem{GulJIA}
V.S. Guliyev,
"Boundedness of the maximal, potential and singular operators in the generalized Morrey spaces,"
\textit{J. Inequal. Appl.} 2009, Art. ID 503948, 20 pp.

\bibitem{GulJMS2013}
V.S. Guliyev,
"Generalized local Morrey spaces
and fractional integral operators with rough kernel,"
\textit{J. Math. Sci.} (N. Y.) {\bf 193} No. 2, (2013) 211--227.

\bibitem{GulAJM2013}
V.S. Guliyev,
"Local generalized Morrey spaces and singular integrals with rough kernel,"
\textit{Azerbaijan Journal of Mathematics,} {\bf 3} (2) (2013), 79--94.

\bibitem{IST14}
T. Iida, Y. Sawano, and H. Tanaka,
"Atomic decomposition for Morrey spaces,"
{\it Z. Anal. Anwendungen} {\bf 33} (2014) 2, 149--170.


\bibitem{JW}
H.~Jia and H.~Wang,
"Decomposition of Hardy-Morrey spaces,"
\textit{J. Math. Anal. Appl.} {\bf 354} (2009), no. 1, 99--110.

\bibitem{Mi} 
T. Mizuhara,
\emph{Boundedness of some classical operators on generalized Morrey spaces}, Harmonic Anal., Proc. Conf., Sendai/Jap. 1990, ICM-90 Satell. Conf. Proc.,  183--189, 1991.

\bibitem{Nakai94}
E.~Nakai,
"Hardy--Littlewood maximal operator,
singular integral operators, and the Riesz potential
on generalized Morrey spaces,"
\textit{Math. Nachr.} {\bf 166}, 1994, 95--103.


\bibitem{Nakai00}
E.~Nakai,
"A characterization of pointwise multipliers
on the Morrey spaces,"
\textit{Sci. Math.} {\bf 3} (2000), 445--454.

\bibitem{NaSa2012}
E.~Nakai and Y.~Sawano,
"Hardy spaces with variable exponents
and generalized Campanato spaces,"
\textit{J.~Funct.~Anal.}, (2012) {\bf 262}, 3665--3748.

\bibitem{NaSa-pre2013}
E.~Nakai and Y.~Sawano,
"Orlicz-Hardy spaces and their duals,"
\textit{Sci. China Math.} {\bf 57} (2014), no. 5, 903--962.

\bibitem{Olsen95}
P.~Olsen,
Fractional integration,
Morrey spaces and Schr\"{o}dinger equation,
\textit{Comm. Partial Differential Equations,} {\bf 20} (1995),
2005--2055.


\bibitem{PlPo37}
M.~Plancherel and M.G.~P\'{o}lya, 
Fonctions enti\`{e}res et int\'{e}grales de fourier multiples, 
(French){\it Comment. Math. Helv. }{\bf 10} (1937), no. 1, 110--163. 

\bibitem{s08-2} Y. Sawano,
"Wavelet characterization of Besov-Morrey and
Triebel-Lizorkin-Morrey spaces,"
\textit{Funct. Approx. Comment. Math.} {\bf 38} (2008), part 1, 93--107.

\bibitem{Sa08-1}
Y.~Sawano,
"Generalized Morrey Spaces for non-doubling measures,"
\textit{Nonlinear Differ. Equ. Appl.,} {\bf 15} (2008), 413--425.

\bibitem{Sa13}
Y.~Sawano,
"Atomic decompositions
of Hardy spaces with variable exponents
and its application to bounded linear operators,"
\textit{Integral Equations Operator Theory } {\bf 77} (2013), no. 1, 123--148.

\bibitem{SaTa1}
Y.~Sawano and H.~Tanaka,
"Morrey spaces for non-doubling measures,"
\textit{Acta Math. Sinica,} {\bf 21} (2005), no. 6, 1535--1544.

\bibitem{SaTa2007}
Y.~Sawano and H.~Tanaka,
"Decompositions of Besov-Morrey spaces and Triebel-Lizorkin-Morrey spaces,"
\textit{Math. Z.} {\bf 257} (2007), no. 4, 871--905.


\bibitem{SaSuTa09-2}
Y.~Sawano, S.~Sugano and H.~Tanaka,
"Identification of the image of Morrey spaces
by the fractional integral operators,"
Proc.~A.~\textit{Razmadze Math.~Institute}, {\bf 149}, (2009), 87--93.


\bibitem{SaSuTa11-1}
Y.~Sawano, S.~Sugano and H.~Tanaka,
"Generalized fractional integral operators
and fractional maximal operators in the framework of Morrey spaces,"
\textit{Trans. Amer. Math. Soc.,} {\bf 363} (2011), no. 12, 6481--6503.

\bibitem{SaSuTa11-2}
Y.~Sawano, S.~Sugano and H.~Tanaka,
"Olsen's inequality and its applications to Schr\"{o}dinger equations,"
\textit{Suurikaiseki K$\hat{o}$ky$\hat{u}$roku Bessatsu} {\bf B26},
(2011) 51--80.

\bibitem{SST12-2}
Y.~Sawano, S.~Sugano and H.~Tanaka,
"Orlicz-Morrey spaces and fractional operators,"
\textit{Potential Anal.,} {\bf 36}, No. 4 (2012), 517--556.

\bibitem{SW}
Y. Sawano and H. Wadade,
"On the Gagliardo-Nirenberg type inequality
in the critical Sobolev-Morrey space,"
\textit{J.~Fourier Anal.~Appl.}, {\bf 19} (2013), no. 1, 20--47.

\bibitem{SSG10}
I.~Sihwaningrum, H.P.~Suryawan and H.~Gunawan,
"Fractional integral operators and Olsen inequalities
on non-homogeneous spaces,"
\textit{Aust. J. Math. Anal. Appl.} {\bf 7} (2010),
no. 1, Art. 14, 6 pp.

\bibitem{Sugano11}
S.~Sugano,
"Some inequalities for generalized fractional integral operators
on generalized Morrey spaces,"
\textit{Math. Inequal. Appl.}, {\bf 14} (2011), no.4, 849--865.

\bibitem{SuTa03}
S.~Sugano and H.~Tanaka,
"Boundedness of fractional integral operators on generalized Morrey spaces,"
\textit{Sci. Math. Jpn.}, {\bf 58} (2003), no. 3, 531--540
(Sci. Math. Jpn. Online {\bf 8} (2003), 233--242).

\bibitem{Stein}
E.M. Stein,
Singular Integrals and Differentiability Properties
of Functions, Princeton Univ. Press, 1970.

\bibitem{Stein1993}
E.M.~Stein,
"Harmonic Analysis, real-variable methods, orthogonality,
and oscillatory integrals,"
\textit{Princeton University Press}, Princeton, NJ, 1993.

\bibitem{Tanaka10}
H.~Tanaka,
"Morrey spaces and fractional operators,"
\textit{J. Aust. Math. Soc.} {\bf 88} (2010), no. 2, 247--259.

\bibitem{TaXu}
L.~Tang and J.~Xu,
"Some properties of Morrey type Besov-Triebel spaces,"
\textit{Math. Nachr } {\bf 278} (2005), no 7-8, 904--917.

\bibitem{Trieb1} H. Triebel,
"Theory of Function Spaces," \textit{Birkhauser,} Basel, 1983.


\bibitem{UNW12-1}
M.I.~Utoyo, T.~Nusantara and B.S.~Widodo,
"Fractional integral operator and Olsen inequality
in the non-homogeneous classic Morrey space,"
\textit{Int. J. Math. Anal.} (Ruse) {\bf 6} (2012), no. 29--32, 1501--1511.


\bibitem{ysy-textbook} W. Yuan, W. Sickel and D. Yang,
"Morrey and Campanato Meet Besov, Lizorkin and Triebel,"
\textit{Lecture Notes in Mathematics,} 2005,
Springer-Verlag, Berlin, 2010, xi+281 pp.


\bibitem{Zo}C. Zorko,
"Morrey space,"
\textit{Proc. Amer. Math. Soc.} {\bf 98} (1986), no. 4, 586--592.


\end{thebibliography}
\end{document}